\documentclass[12pt,oneside]{amsart} 
\usepackage{amssymb,amscd}
\usepackage{euscript}
\usepackage{amsmath,amssymb,amsthm,amsfonts,setspace,hyperref,dsfont,yhmath,euscript}

      \theoremstyle{plain}
      \newtheorem{theorem}{Theorem}[section]
      \newtheorem{lemma}[theorem]{Lemma}
      \newtheorem{corollary}[theorem]{Corollary}
      \newtheorem{proposition}[theorem]{Proposition}
      \newtheorem{remark}[theorem]{Remark}
      
      \newtheorem{definition}[theorem]{Definition}

\numberwithin{equation}{section}

      \makeatletter
      \def\@setcopyright{}
      \def\serieslogo@{}
      \makeatother

\def\M{\mathcal M}
\def\t{\mathcal T}

\def\R{\mathbb R}

\def\Q{\mathbb Q}
\def\Z{\mathbb Z}
\def\N{\mathbb N}
\def\T{\mathbb T}
\def\Td{\mathbb T^d}
\def\a{\alpha}
\def\d{\Delta}

\def\dist{\text{dist}}

\def\Id{\text{Id}}
\def\e{\varepsilon}
\def\Ci{C^\infty}
\def\Cr{C^{N+\alpha}}

\def\xf{{\chi}}
\def\xs{{\chi}}

\def\f{\bar f}

\def\ti{\tilde}

\def\s{{\mathcal{S}}}

\def\E{{\mathcal{E}}}

\def\V{{\mathcal{N}}}

\def\w{{\mathcal{W}}}
\def\v{{\mathcal{V}}}

\def\Wf{\mathcal U}
\def\WL{U}
\def\VL{W}
\def\Vf{\mathcal W}
\def\Uf{\mathcal V}
\def\UL{V}
\def\F{\tilde V}

\def\f{{\mathcal {F}}}
\def\fd{{F}}

\def\p{{\mathcal {P}}}
\def\pd{{P}}

\def\h{{\varphi}}
\def\hd{H}

\def\H{{\mathcal {H}}}

\newcommand{\la}{\lambda}

\def\QED{\hfill\hfill{\square}}

\begin{document}

\date{\today}
\author{Boris Kalinin$^1$ and Victoria Sadovskaya$^2$}

\address{Department of Mathematics, The Pennsylvania State University, University Park, PA 16802, USA.}
\email{kalinin@psu.edu, sadovskaya@psu.edu}

\title [Rigidity of strong and  weak foliations]
{Rigidity of strong and  weak foliations}


\thanks{$^1$  Supported in part by Simons Foundation grant 855238}
\thanks{$^2$ Supported in part by  Simons Foundation grant MP-TSM-00002874}


\begin{abstract}

We consider a perturbation $f$ of a hyperbolic toral automorphism $L$.  We study rigidity related to exceptional properties of the strong and weak stable foliations for $f$. If the strong foliation  is mapped to the linear one by the conjugacy $h$ between $f$ and $L$, we  obtain smoothness of $h$ along the weak foliation and regularity of the joint foliation of the strong and unstable foliations. We also establish a similar global result. If the weak foliation is sufficiently regular, we  obtain smoothness of the conjugacy along the strong foliation and regularity of the joint foliation of the weak and unstable foliations. If both conditions hold, we get smoothness of $h$ along the stable foliation. We also deduce a rigidity result for the symplectic case. The main theorems are obtained in a unified way using our new result on relation between holonomies and normal forms.

\end{abstract}

\maketitle 


\section{Introduction}


In this paper we consider a perturbation $f$ of a hyperbolic toral automorphism $L$.  We study rigidity related to exceptional properties of the strong and weak stable foliations for $f$.
We recall that an automorphism $L$ of $\T^d$ is hyperbolic if the matrix  has no eigenvalues of modulus 1. We denote its stable and unstable subspaces by $E^s$ and $E^u$.

Let $f$ be a $C^\infty$ diffeomorphism 
 of $\T^d$ which is $C^1$ close to $L$. Then $f$ is an Anosov diffeomorphism, i.e., the tangent bundle of $\T^d$ splits into a $Df$-invariant direct sum of the stable and unstable subbundles $\E^s$ and $\E^u$, where 
 $$
 \|Df|_{\E^s}\|<1\;\text{  and  }\;\|Df^{-1}|_{\E^u}\|<1
 $$ 
 for some Riemannian metric.
Also, $f$  is conjugate to $L$ by a homeomorphism $h$, i.e., 
\begin{equation}\label{h}
h\circ f=L\circ h.
\end{equation}
The conjugacy close to the identity  is unique, 
and it is bi-H\"older but usually not $C^1$. Any two conjugacies differ by an affine automorphism of $\T^d$ commuting with $L$, and hence they have the same regularity.

We denote the stable and unstable  foliations for $L$ by $W^{s}$ and $W^{u}$, and for $f$ by $\w^{s}$ and $\w^{u}$. The foliations  $\w^s$ and $\w^u$ have uniformly $C^\infty$ leaves, but they are not even $C^1$ foliations  in general.
A foliation is  $C^r$ if it has $C^r$ local foliation charts. We say that the leaves are uniformly 
$C^r$ if locally they can be $C^r$ embedded with the embeddings varying continuously in the $C^r$ topology. 
For a foliation $\w$ with uniformly $C^r$ leaves, we say that a function on $\T^d$ is uniformly $C^r$
along  $\w$ if locally its restrictions to the leaves are $C^r$ and vary continuously in the $C^r$ topology,
see Section \ref{sec NF fol} for more details and noninteger $r$.

We recall that the Lyapunov space for an exponent $\chi$ of $L$  is the sum of generalized 
eigenspaces of all eigenvalues of $L$ with modulus $e^\chi$.
If $L$ has more than one stable Lyapunov exponent, one can take a dominated splitting
$E^s= E^{ss}\oplus E^{ws}$ into strong and weak parts by combining
one or more Lyapunov spaces into each part.
Since $f$ is $C^1$ close to $L$, its stable subbundle   has the corresponding dominated splitting
$C^0$-close to that of $L$:
 \begin{equation}\label{ss ws for f}
\E^s= \E^{ss}\oplus \E^{ws}. 
\end{equation}
The bundle $\E^{ss}$ is integrable to the strong foliation $\w^{ss}$, which has uniformly $C^\infty$ leaves
and is a $C^\infty$ subfoliation of the leaves of $\w^s$.
The bundle $\E^{ws}$ is also integrable to the weak foliation  $\w^{ws}$,  but  in general even the leaves of $\w^{ws}$ are only uniformly $C^{1+\text{H\"older}}$. We always have $h(\w^{s})=W^{s}$  and $h(\w^{ws})=W^{ws}$ by \cite{G08}, but usually not $h(\w^{ss})=W^{ss}$.  Our main goal is to explore what rigidity properties follow if we assume that $h(\w^{ss})=W^{ss}$  or that  $\w^{ws}$ is sufficiently regular.

\vskip.1cm 

Let $h$ be the conjugacy \eqref{h} between $f$ and $L$ close to the identity. Then it can be written as $h(x)=x+\d (x)$, where $\d: \T^d \to \R^d$.
We will consider the components $\d^*$ of $\d$, where $*=u, s,ss,ws$, with respect to the splittings
$$
\R^d=E^u\oplus E^s= E^u\oplus E^{ss}\oplus E^{ws}.
$$
We will also write $h^*(x)=x^*+\d ^*(x)$, which can be defined  globally for the lift of $h$ to the universal cover $\R^d$, or locally on $\T^d$. Thus the regularity of $h^*$ is well-defined.
\vskip.1cm

First we consider the rigidity of a strong foliation. 

\begin{theorem} [Rigidity of a strong foliation] \label{th strong} 
Let $L$ be a hyperbolic automorphism of $\,\T^d$ with dense leaves of $W^{ss}$. 
Let $f$ be a $C^\infty$ diffeomorphism sufficiently $C^1$ close to $L$, and let $h$ be a topological conjugacy between $f$ and $L$. Then for the statements below we have 
$$
  (1)\iff (2)\iff (3) \implies (5) \implies (4).
$$
Moreover, if the leaves of each of the Lyapunov subfoliations  of $W^{ws}$ are dense in $\T^d$, then the five statements are equivalent.

 \vskip.1cm
 \begin{itemize} 
 \item[(1)]\, $h(\w^{ss})=W^{ss}$,
 \vskip.1cm
 \item[(2)]\, $\w^{u}$ and $\w^{ss}$ are jointly integrable to a  foliation $\w^{u+ss}$, 
 \vskip.1cm
 \item[(3)]\, (2) holds and the foliation  $\w^{u+ss}$   is conjugate to the linear foliation $W^u\oplus W^{ss}$  by a 
 \,$C^{\infty}$ diffeomorphism of $\T^d$,
 \vskip.1cm
 \item[(4)]\, $h$ is a $C^{1}$ diffeomorphism along the leaves of $\w^{ws}$,  and \\
\,the derivative 
$D(h|_{\w^{ws}(x)}):\,\E^{ws}_x \to \R^d $ is H\"older continuous on $\T^d$,
 \vskip.1cm
 \item[(5)]\,  $h^{ws}$ is $C^{\infty}$ on $\T^d$,\,   and   if the leaves of $\w^{ws}$ are uniformly $C^q$,\\
  then $h$ is a uniformly $C^{q}$ diffeomorphism along $\w^{ws}$.
\end{itemize}
\end{theorem}

In Section \ref{examples} we give basic examples of perturbations $f$ satisfying the assumptions
of Theorem \ref{th strong}, illustrating the regularity we obtain and necessity of the extra assumption
for the implications (4)$\implies$(5) and (4)$\implies$(1). Similar examples for subsequent theorems are also included.
Our results essentially recover the main features of the basic models.

\begin{remark} [Irreducibility conditions]
A matrix $L\in SL(d,\Z)$ is called {\em irreducible} if it has no nontrivial rational invariant subspaces or, equivalently, if its characteristic polynomial is irreducible over $\Q$. 
If $L$ is irreducible, or more generally, weakly irreducible, then all Lyapunov foliations for $L$ have dense leaves, and hence  the five statements in Theorem \ref{th strong} are equivalent.
We define and discuss weak irreducibility in Section \ref{irreducibility}.
\end{remark}

Joint integrability of $\w^{u}$ and $\w^{ss}$ plays an important role in the study of ergodic properties of foliation $\w^{ss}$,  and it is also related to the Lyapunov exponents on $\E^{ws}$.  It has been extensively studied, primarily on $\T^3$, see e.g. \cite{A1,GaSh,DR}. In higher dimensions, its relation to rigidity was considered by Gogolev and Shi in \cite{GSh}. 
They proved that  $(1)\iff (2)$ in general, and that $(2)\iff (4)$ under the assumption that $L$ is irreducible and 
has at most two-dimensional Lyapunov spaces.  Our assumptions on $L$ are  considerably weaker and we
obtain stronger conclusions (3) and (5). Our approach is completely different, it yields higher smoothness directly and does not rely on \cite{GSh} aside from $(1)\iff (2)$. Our techniques are mostly global, but require narrow spectrum of $Df|_{\E^{ws}}$ to use normal forms. 

Using our techniques and narrow spectrum for $Df|_{\E^{ws}}$ obtained in \cite{GSh} we estalish the following global version of Theorem \ref{th strong}. The  bunching assumption \eqref{bunch} in the theorem means that nonconformality of $Df$ on $ \E^{ws} $ is dominated by the expansion on $ \E^{u} $, and is trivially satisfied if $\dim \E^{ws} =1.$
 A splitting $\E^s= \E^{ss}\oplus \E^{ws} $ is called {\em absolutely dominated} if there exists $0<\rho<1$
such that with respect to some continuous family of Riemannian  norms  on $\E^s$ we have
$$
\|Df(u)\| < \rho < \|Df(v)\|\;\text{ 
for all unit vectors $u\in \E^{ss}$ and $v\in \E^{ws}$.}
$$
This condition automatically holds  in the setting of Theorem \ref{th strong}.

\begin{theorem} [Global rigidity of a strong foliation] \label{global strong} 
Let $L$ be an irreducible hyperbolic automorphism of $\,\T^d$ 
and let $f$ be a $C^\infty$ Anosov diffeomorphism of $\,\T^d$ conjugate to $L$ by a homeomorphism $h$. 
Suppose that  $f$  has an absolutely dominated splitting $\E^s= \E^{ss}\oplus \E^{ws} $
and satisfies the bunching condition
\begin{equation}\label{bunch}
\|Df|_{\E^{ws}(x)}\| \cdot \|(Df|_{\E^{ws}(x)})^{-1}\|  \cdot \|(Df|_{\E^{u}(x)})^{-1}\| <1\;\;\text{ for all $x\in \T^d$}.
\end{equation}
Then  the statements (1), (2), (3$\,'$), (4), (5) are equivalent, where (1), (2), (4), (5) are as 
 in Theorem \ref{th strong} and (3$\,'$) is 
 \vskip.1cm
 \begin{itemize} 
  \item[(3$'$)]\, $\w^{u}$ and $\w^{ss}$ are jointly integrable to a $C^{\infty}$ foliation $\w^{u+ss}$\\ whose holonomies preserve a H\"older continuous Riemannian metric on $\E^{ws}$.
 \end{itemize}
\end{theorem}

Our approach also allows us to obtain  rigidity  related to regularity of the weak foliation. We recall that $h(\w^{ws})=W^{ws}$, and hence  $\w^u$ and $\w^{ws}$ are jointly integrable  to the foliation $\w^{u+ws}=h^{-1}(W^u\oplus W^{ws})$ with uniformly $C^{1+\text{H\"older}}$ leaves. 

Let $r_{ss}(L)$ be the ratio of the top and bottom Lyapunov exponents of $L$ on $E^{ss}$, i.e., 
\begin{equation}\label{rss}
r_{ss}(L)= (\log \rho_{\min}) / (\log {\rho_{\max}})\ge 1,
\end{equation}
where $0<\rho_{\min}\le \rho_{\max}<1$ are the smallest and the largest moduli of the eigenvalues of $L$ on $E^{ss}$.

\begin{theorem} [Rigidity of a weak foliation] \label{th weak fol}
Let $L$ be a hyperbolic automorphism of $\,\T^d$ with dense leaves of $W^{ws}$. 
Let $f$ be a $C^\infty$ diffeomorphism sufficiently $C^1$ close to $L$, and let $h$ be a topological conjugacy between $f$ and $L$. If $r>r_{ss}(L)$ and $r\notin \N$, then  the following are equivalent.
 \vskip.1cm
 \begin{itemize} 
 \item[(1)]\,  $\w^{ws}$ is a uniformly $C^{r}$ subfoliation of $\w^s$,
  \vskip.1cm 
 \item[(2)]\,\,  $h^{ss}$  is a uniformly $C^\infty$ diffeomorphism along $\w^{ss}$,\,
  and  $h^{ss}$ is $ C^{r}$ on $\T^d$, 
 \vskip.1cm
  \item[(3)]\, The joint foliation $\w^{u+ws}$ is conjugate to  the linear foliation $W^{u} \oplus W^{ws}$ by a $C^{r}$ diffeomorphism. 
  \end{itemize}
 \vskip.1cm
 
\noindent If in addition $h(\w^{ss})=W^{ss}$, \, then\, \em{(1,\,2,\,3)} $\implies$ $h$  is uniformly $C^{\infty}$ along $\w^{ss}$.

\end{theorem}

This theorem follows from a more general technical result, Theorem \ref{th weak hol}, where we assume only  regularity of the holonomies of $\w^{ws}$ between the leaves of $\w^{ss}$, rather than regularity of the subfoliation. The distinction is not assuming higher regularity of the leaves of $\w^{ws}$, which in general are 
only $C^{1+\text{H\"older}}$.

\vskip.3cm

Combining rigidity of the weak and strong subfoliations, Theorems \ref{th strong} and \ref{th weak fol}, we obtain the following characterizations of smoothness of the conjugacy along $\w^s$.

\begin{theorem}  [Rigidity of strong and weak foliations]\label{strong+weak}  
Let $L$ be a hyperbolic automorphism of $\,\T^d$.
Suppose that $W^{ws}$ and $W^{ss}$ have dense leaves. Let $f$ be a $C^\infty$ diffeomorphism sufficiently $C^1$ close to $L$, and let $h$ be a topological conjugacy between $f$ and $L$.
Then the following are equivalent. 
\vskip.1cm 
\begin{itemize}
\item[(1)]\,  $h(\w^{ss})=W^{ss}$ and $\w^{ws}$ is a uniformly $C^\infty$ subfoliation of $\w^s$,
 \vskip.1cm
\item[(2)]\,  $h$ a is uniformly $C^\infty$ diffeomorphism along $\w^{s}$,
 \vskip.1cm
\item[(2$'$)] $h^s$  is in  $C^\infty(\T^d)$ and a diffeomorphism along $\w^{s}$,
 \vskip.1cm
\item[(3)] $\w^{u}$ is conjugate to the linear  foliation $W^u$ by a $C^\infty$ diffeomorphism.
\end{itemize}
\end{theorem}
We note that the results above do not yield bootstrap of regularity of $h$ from $C^1$ to $C^\infty$
without knowing higher regularity of the leaves of $\w^{ws}$. However, we obtain a stronger bootstrap
for symplectic $L$ and $f$.
\vskip.3cm

\noindent {\bf Symplectic case.} Now we apply the above results to obtain rigidity results for symplectic $L$ and $f$. 

 Let $L$ be a symplectic hyperbolic automorphism of $\T^d$ with dominated splittings 
\begin{equation}\label{symp split}
E^s= E^{ss}\oplus E^{ws}\; \text{ and }\; E^u= E^{uu}\oplus E^{wu}
\end{equation}
such that  $\dim E^{ss}=\dim E^{uu}$ and hence $\dim E^{ws}=\dim E^{wu}$. Here $E^{uu}$ and $E^{wu}$ are strong and weak unstable subbundles.
Let $f$ be a $C^1$ small perturbation of $L$.
Then $f$ has corresponding dominated splittings
$$
\E^s= \E^{ss}\oplus\E^{ws}\; \text{ and }\; \E^u= \E^{uu}\oplus\E^{wu}.
$$
The subbundle $\E^{ws+wu}$ is integrable with $C^{1+\text{H\"older}}$ leaves,  and $h(\w^{ws+wu})=W^{ws+wu}$.

\begin{theorem} [Symplectic rigidity] \label{symp}
Let $L$ be a  symplectic hyperbolic automorphism of $\,\T^d$.
Suppose that the foliations $W^{ws}$, $W^{ss}$, $W^{wu}$, and $W^{uu}$ have dense leaves. Let $f$ be a $C^\infty$ diffeomorphism sufficiently $C^1$ close to $L$ and preserving a $C^\infty$ symplectic form.
Let $h$ be a topological conjugacy between $f$ and $L$. Then for the statements below we have
$$
(1) \iff (2) \iff (3) \implies (4).
$$
\noindent Moreover, if the leaves of each of the Lyapunov subfoliations  of $W^{ws+wu}$ are dense in $\T^d$, then the four statements are equivalent.
\vskip.1cm 
\begin{itemize}
\item[(1)]\, $h$ is a $C^\infty$ diffeomorphism,
\vskip.1cm 
\item[(2)]\,  $h(\w^{ss})=W^{ss}$ and  $\,h(\w^{uu})=W^{uu}$, 
\vskip.1cm 
\item[(3)]\,  $h$ is $\a$-H\"older with $\a$ sufficiently close to 1, 
\vskip.1cm

\item[(4)]\, $h$ is a $C^{1}$ diffeomorphism along the leaves of $\w^{ws+wu}$,  and 
\,the derivative 
$D(h|_{\w^{ws+wu}(x)}):\,\E^{ws+wu}_x \to \R^d $ is H\"older continuous on $\T^d$.
\end{itemize}
\end{theorem}

\noindent Preservation of a symplectic form is used only for symplectic orthogonality of 
$\E^{ws}\oplus\E^{wu}$ to $\E^{ss}\oplus\E^{uu}$, which allows us to deduce smoothness
of the former from that of the latter.

Condition (2) in Theorem \ref{symp} implies joint integrability of $\w^{uu}$ with $\w^{ss}$. 
It is an open question whether the converse holds. For an irreducible $L$ with {\it one-dimensional}
weak foliations  $W^{wu}$ and $W^{ws}$, Gogolev and Shi showed 
 in the proof of \cite[Theorem 6.1]{GSh} that joint integrability of $\w^{uu}$ and $\w^{ss}$ implies (2). 
Using this result together with the implication $(2) \implies (1)$ of Theorem \ref{symp} we obtain
 the following corollary. It extends \cite[Theorem 1.4]{GSh} from $d=4$ to any $d\ge 4$.
 
\begin{corollary} \label{symp cor}
Let $L$ be an irreducible symplectic hyperbolic automorphism of $\,\T^d$ with 
splitting \eqref{symp split} such that $\dim E^{ws}= \dim E^{wu}=1$.  Let $f$ be a $C^\infty$ diffeomorphism sufficiently $C^1$ close to $L$ and preserving a $C^\infty$ symplectic form.
If the strong foliations $\w^{uu}$ and $\w^{ss}$ are jointly integrable then $f$ is $C^\infty$ conjugate to $L$.
\end{corollary}

\begin{remark} [Finite regularity]Our results hold if $f$ is a $C^t$ diffeomorphism rather than $C^\infty$, 
provided that $t$ is sufficiently large, and  with the  $C^\infty$ regularity of other objects replaced
by  $C^{t-\delta}$ for any $\delta >0$. Theorems \ref{th strong} and \ref{global strong} 
require $t>r_{ws}$ defined similarly to \eqref{rss}. In particular, if $L$ has one Lyapunov exponent
on $E^{ws}$ then any $t>1$ suffices. Theorem \ref{th weak fol}  requires $t>r$,  
Theorem \ref{strong+weak} requires $t>\max\{r_{ws},r_{ss}\}$, 
Theorem \ref{symp} requires $t>\max\{r_{wu},r_{uu},r_{ws},r_{ss}\}$, and 
Corollary \ref{symp cor} requires $t>\max\{r_{uu},r_{ss}\}$. These are the regularities needed to apply the results on normal forms. The proofs work without any significant modifications.
\end{remark}

\vskip.2cm 
\noindent{\bf Normal forms and holonomies.} The proofs of the theorems above rely on our new results on normal forms and holonomies, which are of independent interest.  We give preliminaries on normal forms  
in Section \ref{sec NF prelim} and formulate and prove the new results in Section \ref{sec NF fol}. 
In the context of two invariant transverse subfoliations $\w$ and $\v$ of $\w^s$ we prove in Theorem \ref{Main Hol} 
that if $Df|_{T\w}$ has narrow spectrum  and the holonomies of $\v$ between  $\w$ are sufficiently 
smooth then they preserve normal forms, i.e., they are sub-resonance polynomials when written
 in normal form coordinates. This allows us to use holonomies along  $\v$ together with
density of its leaves to obtain regularity of the conjugacy along $\w$. Prior results using normal forms 
with holonomies were in the context of {\em neutral} foliations \cite{FKS,GKS23} and they do not apply to expanding or contracting foliations. We use our results
with both $\w=\w^{ss}$ and $\w=\w^{ws}$. In the latter case, even existence of normal forms is 
nontrivial since the leaves  of  $\w^{ws}$ have low regularity. We overcome this problem
 in Theorem \ref{NF weak} by using smooth holonomies of $\v=\w^{ss}$.

\vskip.2cm 
\noindent{\bf Structure of the paper.} 
In Section \ref{ex+irred} we give examples to illustrate the main theorems, and discuss irreducibility.
In Section \ref{sec NF prelim} we give preliminaries on normal forms, and in Section \ref{sec NF fol} formulate and prove the new results. We prove Theorem \ref{th strong}  in Section \ref{hsmoothWs},  Theorem \ref{global strong}    in Section \ref{proof global strong},  Theorem \ref{th weak fol}  in Section \ref{proof weak fol},  Theorem \ref{strong+weak}  in Section \ref{proof strong+weak}, and Theorem \ref{symp}  in Section \ref{proof symp}.


\section{Examples and weak irreducibility} \label{ex+irred}

\subsection{Examples}\label{examples} $\;$
\vskip.1cm 
\noindent {\bf (i)} We illustrate Theorem \ref{th strong} in a basic setting of a hyperbolic  automorphism of $\T^3$, which is always irreducible. 
Let $L$ be a hyperbolic automorphism of $\T^3$ with eigenvalues  
$$0<\la_{ss}<\la_{ws}<1<\la_u,
$$ and let $e_{ss}$, $e_{ws}$, and $e_u$ be corresponding unit eigenvectors.  We consider a $C^1$ small perturbation $f$ of $L$ of the form 
$$ f(x)=L(x)+\varphi_u(x)e_u+\varphi_{ss}(x)e_{ss},
$$
where $\varphi_u$ and $\varphi_{ss}$ are smooth real-valued functions on $\T^3$.

Then a conjugacy between $f$ and $L$ can be found in the form 
$$ h(x)=x+\phi_u(x)e_u+\phi_{ss}(x)e_{ss}.
$$
One can take smooth functions $\varphi_u$ and $\varphi_{ss}$ for which the corresponding functions 
$\phi_u$ and $\phi_{ss}$ are not smooth, for example trigonometric polynomials as in \cite[Theorem 6.3]{L1} and \cite[Theorem 5.5.5 and Remark 5.5.6]{KtN}.

For such a perturbation, the linear foliation $W^{u+ss}$ is preserved by $f$ and $h$, and hence $\w^{u+ss}=W^{u+ss}=h(\w^{u+ss})$.\, While $\w^{ss}\ne W^{ss}$, we have $h(\w^{ss}) = W^{ss}$ since $\w^{ss}=\w^{u+ss}\cap \w^s$ and $h(\w^s)=\w^s$. We see that  while $h$ is not smooth,  $h(x)^{ws}=x^{ws}$ and hence is smooth on $\T^d$. One can also see by differentiating term-wise that 
$\phi_{ss}(x)= -\sum _{k=1}^\infty \la_{ss}^{k-1}\varphi (f^{-k}x)$ has one derivative in the direction of $e_{ws}$.
This shows that $h$ is $C^1$ along $\w^{ws}$, which corresponds to (4) in Theorem \ref{th strong}.
However, the series for higher derivatives may diverge, and the bundle $\E^{ws}$ is only
H\"older in general.
\vskip.15cm

\noindent {\bf (ii)} One can modify this example to show that the implications  (4)$\implies$(5) and (4)$\implies$(1) in Theorem \ref{th strong}
can fail without density of the Lyapunov leaves. We take an automorphism 
$$L=L_1 \times L_2\quad\text{of }\;\T^5=\T^3\times \T^2,
$$ 
where $L_1$ is as in (i), and $L_2$ is an automorphism of $\T^2$ with  eigenvalues $0< \mu_s <1 <\mu_u=\mu_s ^{-1}$
and unit eigenvectors $v_s$ and $v_u$. If $\mu_s < \la_{ws}$, we can take $W^{ss}$ for $L$ as the span of $e_{ss}$ and $v_s$ and $W^{ws}$ as the span of $e_{ws}$. We consider the perturbation 
$$f(x_1,x_2)=(L_1(x_1)+\varphi(x_2)e_{ws},\,L_2(x_2)).
$$
Then the conjugacy can be written as $h(x_1,x_2)=(x_1+\phi(x_2)e_{ws},\,x_2)$. 
Thus for a fixed $x_2$ the conjugacy is a translation on $\T^3$ in $W^{ws}$ direction, and hence (4) is satisfied.
If $\varphi$ is taken so that $\phi$ is not $C^1$ on $\T^2$ then  component $h^{ws}$ is not $C^1$ on $\T^3\times \T^2$. Thus (5) fails and hence (1) has to fail too since (1)$\implies$(5). It can also be directly
computed as in \cite[Remark 5.5.6]{KtN}  that $h(\w^{ss})\ne W^{ss}$.
\vskip.2cm 

\noindent {\bf (iii)} In the setting of (i), we can similarly obtain perturbations satisfying the assumptions of the theorems below.
Specifically, 
$$ 
\begin{aligned}
&f(x)=L(x)+\varphi_{u}(x)e_{u}\quad\text{for Theorem \ref{strong+weak},}\\
&f(x)=L(x)+\varphi_{ws}(x)e_{ws}\quad\text{for Theorem \ref{th weak fol},\, and}\\
&f(x)=L(x)+\varphi_u(x)e_u+\varphi_{ws}(x)e_{ws}\quad\text{for Theorem \ref{th weak hol}.}
\end{aligned}
$$
These examples illustrate the conclusions we get in these theorems.

\subsection{Irreducibility and weak irreducibility}\label{irreducibility}$\;$
\vskip.1cm 

We recall that $L\in SL(d,\Z)$ is {\em irreducible} if it has no nontrivial rational invariant subspaces or, equivalently, if its characteristic polynomial is irreducible over $\Q$. The eigenvalues of an irreducible $L$ are simple. Irreducibility of $L$ implies that any $L$-invariant linear foliation of $\T^d$ is dense in $\T^d$. 

A weaker assumption on $L$, called {\em weak irreducibility}, gives denseness of every Lyapunov foliation for $L$. It was introduced and discussed in  \cite{KSW23}, see Section 3.3 there. Let $\rho_1, \dots,\rho_m$ be distinct moduli of eigenvalues of $L$ and let $E^{\rho_1}, \dots, E^{\rho_m}$ be the corresponding Lyapunov subspaces. We say that $L$ is {\em weakly irreducible} if for each $i$ the space $\hat E^{\rho_i}=\oplus_{j\ne i} E^{\rho_j}$ contains no non-zero elements of $\Z^d$. Equivalently, there is a set $S\subset \R$ so that for each irreducible over $\Q$ factor of the characteristic polynomial of $L$ the set of moduli of its roots equals $S$. A weakly irreducible $L$ is not necessarily diagonalizable.

 \begin{lemma} [Weak irreducibility] \label{weak irred}
A matrix $L\in SL(d,\Z)$ is weakly irreducible if and only if for each Lyapunov foliation of $L$ the 
leaves are dense in $\T^d$.
\end{lemma}

 \begin{proof}
 We will prove the equivalence to the second condition above. 
 Let $p_L$ be the characteristic polynomial of $L$ and $p_L=\prod_{k=1}^K p_k^{d_k}$ 
 be its prime decomposition over $\Q$. 
 
 Let $S=\{\rho_1, \dots, \rho_m\}$ be the set of moduli of roots for each $p_k$. Suppose that for some $i$
 the leaves of the Lyapunov foliation of $W^{\rho_i}$ are not dense in $\T^d$. Let $E$ be the minimal  
 rational $L$-invariant subspace containing $E^{\rho_i}$. Then $E \ne \R^d$. We consider the restriction $B=L|_E$ and the induced
 operator $C$ on the quotient $\R^d/E$. Then we have $p_L=p_B \cdot p_C$,
 and all three are rational polynomials. Hence $p_C=\prod_{k=1}^K p_k^{c_k}$, and so by the assumption 
 $p_C$ has at least one root of modulus $\rho_i$. Then $p_L$ has more roots of modulus 
 $\rho_i$, counted with multiplicities, than the corresponding number for $p_B$. 
 This is impossible since both should be dim$E^{\rho_i}$, as $E^{\rho_i}\subset E$.
 
To prove the converse, assume that for each Lyapunov foliation of $L$ the leaves are dense in $\T^d$.
Let $S=\{\rho_1, \dots, \rho_m\}$ be the set of moduli of roots of $p_L$, and suppose that 
for some $\rho_i\in S$ and $k_0\in \{1,\dots , K\}$ no root of $p_{k_0}$ has modulus $\rho_i$.
Let $\R^d = \oplus V_k$ be the  splitting into rational  $L$-invariant subspaces $V_k =\ker p_k^{d_k}(L)$.
As the eigenvalues of $L|_{V_k}$ are the roots of $p_k$, we nave $E^{\rho_i} \cap V_{k_0} =0$.
This implies that $E^{\rho_i} \subset \oplus_{k\ne {k_0}} V_k$. Indeed $E^{\rho_i}= \oplus_k (V_k \cap E^{\rho_i})$ since the splittings  $\oplus_j  E^{\rho_j}$ and  $\oplus _k V_k $ have a common refinement.
Since $\oplus_{k\ne {k_0}} V_k$ is a  rational $L$-invariant subspace smaller than $\R^d$,
we conclude that the leaves of corresponding Lyapunov foliation $W^{\rho_i}$ are not dense in $\T^d$.
\end{proof}


\section{Preliminaries on normal forms} \label{sec NF prelim}

\subsection{Smooth extensions and sub-resonance  polynomials} \label{polynomials}

Let $\E$ be a continuous vector bundle over a compact metric space $\M$, let  $\V$ be
 a neighborhood of the zero section in $\E$, and let $f$ be a homeomorphism of $\M$. 
We consider an extension $\f : \V \to \E$ that projects to $f$ and preserves the zero section.
We assume that the corresponding fiber maps $\f_x: \V_x \to \E_{fx}$ are $C^r$ diffeomorphisms.
We will consider $r>1$,
and for $r \notin \N$ we will understand $C^r$ in the usual sense that the derivative of order 
$N=\lfloor r \rfloor$ is H\"older with exponent $\a=r-\lfloor r \rfloor$.

We assume that the fibers $\E_x$ are equipped with a continuous family of Riemannian norms.
 We denote by $B_{x,\sigma }$ the closed ball of radius $\sigma >0$ centered at 
$0 \in \E_x$. For $N \in \N$ and $0\le \a <1$ we denote by $\Cr (B_{x,\sigma })=\Cr (B_{x,\sigma}, \E_{fx})$ 
the space of functions $R: B_{x,\sigma}\to \E_{fx}$ with continuous derivatives up to order 
$N$ on $B_{x,\sigma}$ and, if $\a>0$, with $\a$-H\"older $N^{th}$ derivative at $0$.

  \vskip.1cm

\begin{definition} \label{Cr ext} 
We say that $\f$ is a {\em $\Cr$ extension of $f$}, where $N \in \N$ and $0\le\a <1$,
if for some $\sigma>0$ the fiber maps $\f_x:  B_{x,\sigma} \to \E_{fx}$ are $\Cr$ diffeomorphisms which depend continuously on $x$ in $C^N$ topology with uniformly bounded norms $\|\f _x\|_{\Cr (B_{x,\sigma })}$.
  \vskip.05cm
We say that $\h =\{\h_x\}_{x\in X}$, where $\h_x:  B_{x,\sigma} \to \E_{fx}$, is a {\em $\Cr$ coordinate change} if it is a $\Cr$ extension of the identity map on $\M$ preserving the zero section.
\end{definition}

For a smooth extension $\f$ we will denote by $F$ its derivative   at the zero section, i.e.,
$F : \E \to \E$ is a continuous linear extension of $f$ whose fiber maps are linear isomorphisms
$F_x =D_0 \f_x  : \E_x \to \E_{fx}$. 

\begin{definition} \label{chi ext} 
Let $\e>0$ and let 
\begin{equation}\label{chi}
\chi = (\chi_1, \dots, \chi_\ell), \;\text{ where }\,\chi_1<\dots<\chi_\ell<0.
\end{equation}
We say that a linear extension $F$ has {\em $(\chi,\e)$-spectrum}
if there is a continuous $F$-invariant splitting 
\begin{equation} \label{splitting}
\E=\E^{1} \oplus \dots \oplus \E^{\ell}
\end{equation} 
and a  continuous family of Riemannian norms $\|.\|_{x}$ on $\E_x$ such 
that for $i=1, \dots, \ell$,
 \begin{equation} \label{estAEi}
  e^{\chi _i -\e} \| t \|_{x} \le  \| \fd _x (t)\|_{fx} \le e^{\chi _i +\e} \| t\|_{x}
  \quad\text{for every }\, t \in \E^i_x
\end{equation}
and the splitting \eqref{splitting} is orthogonal.
\end{definition}
 
We say that a map between vector spaces is {\em polynomial}\, if each component  is given by a polynomial 
in some, and hence every, basis.
We will consider a polynomial map $P: \E_x \to \E_y$ with $P(0_x)=0_y$ and 
split it into components $(P_1(t),\dots,P_{\ell}(t))$, where $P_i: \E_x \to \E_y^i$. 
Each $P_i$ can be written uniquely as a linear combination of  polynomials
of specific homogeneous types.  A map $Q: \E_x \to \E_y^i $ has {\em homogeneous type} 
$s= (s_1, \dots , s_\ell)$, where $s_1,\dots,s_{\ell}\,$ are non-negative integers, 
if for any real numbers  $a_1, \dots , a_\ell$ and vectors
$t_j\in \E_x^j$, $j=1,\dots, \ell,$ we have 
\begin{equation}\label{stype}
Q(a_1 t_1+ \dots + a_\ell t_\ell)= a_1^{s_1} \cdots  a_\ell^{s_\ell} \cdot Q( t_1+ \dots + t_\ell).
\end{equation}

\begin{definition} \label{SRdef}
We say that a homogeneous type $s= (s_1, \dots , s_\ell)$ for $Q: \E_x \to \E_y^i$  is  
\begin{equation}\label{sub-resonance}
\text{{\bf sub-resonance} if} \quad \chi_i \le \sum_{j=1}^\ell s_j \chi_j.
\end{equation}
We say that a polynomial map $P: \E_x \to \E_y$ is {\em sub-resonance}
if each component $P_i$ has only terms of  sub-resonance  homogeneous types.
We denote by $\s_{x,y}$ the set of all sub-resonance  polynomials 
$P: \E_x \to \E_y$ with $P(0)=0$ and invertible derivative at $0$.
\end{definition}

Clearly, for any sub-resonance relation we have $s_j=0$ for $j<i$ and $\sum s_j \le \chi_1/ \chi_\ell$.
Hence sub-resonance polynomials have degree at most $d(\chi)= \lfloor \chi_1/\chi_\ell \rfloor$.

We will denote  $\s_{x,x}$ by $\s_x$, which is a finite-dimensional Lie group
 with respect to the composition if $\e>0$ is sufficiently small \cite{GK}. 
Any map $P\in \s_{x,y}$ induces an isomorphism between $\s_x$ and $\s_y$ by conjugation. 
In particular, this holds for any invertible linear map which respects the splitting \eqref{splitting}.
\vskip.1cm

\subsection{Normal forms for contracting extensions }\label{ssec NFext}

The following theorem was established in \cite{GK,G} for $r \in \N \cup \{\infty\}$, 
 and in \cite{K24} for this setting.

For a given $\chi $ as in \eqref{chi}, there is $\e_0=\e_0 (\chi)>0$  given by  \cite[(3.13)]{K24} which ensures that the spectrum is sufficiently narrow and will suffice for all results and proofs below.

\begin{theorem}\cite[Theorem 4.3]{K24} \, 
{\em (Normal forms for contracting extensions)}\label{NFext} $\;$ \\
Let $f:X\to X$ be a homeomorphism of a compact metric space $\M$, let
$\E$ be a continuous vector bundle over $\M$,  let  $\V$ be a neighborhood of the zero section in $\E$, and let $\f:\V  \to \E$ be a $C^r$ extension of $f$ that preserves the zero  section.  
 Suppose that the derivative $F$ of  $\f$ at the zero section 
  has $(\chi,\e)$-spectrum with   $\chi_1/\chi_\ell <r$ and $\e<\e_0$. 
\vskip.05cm

Then there exists a $C^r$ coordinate change 
$\h=\{ \h_x\}_{x\in X}$ with diffeomorphisms   
 $\h_x : B_{x,\sigma} \to \E_x$  satisfying $\h_x(0)=0$ and $D_0 \h_x =\Id \,$
 which conjugates $\f$ to a continuous polynomial extension $\p: \E  \to \E$ of $f$ of sub-resonance type,
 i.e.,
 \begin{equation}\label{nfs}
\h_{fx} \circ \f_x =\p_x \circ \h_x, \; \text{ where }  \;  \p_x\in \s_{x,fx}
 \; \text{ for all } x\in X. 
\end{equation}
\vskip.1cm
If $\f$  is a $\Ci$ extension then the coordinate 
change $\h$ is also $\Ci $.
\end{theorem}

\noindent Any two such coordinate changes $\h_x$ and $\h_x'$ satisfy $\h_x'=\h_x \circ g_x$ 
for some $g_x \in  \s_x$.


\section{Normal forms for contracting foliations} \label{sec NF fol}

Let $f$ be a $C^r$ diffeomorphism of a compact manifold  $\M$, where $r=N+\alpha$. 
We will consider an $f$-invariant 
continuous foliation $\w$ of $X$ with {\em uniformly $C^r$ leaves}. By this we mean that for some $R>0$
the balls $B^\w(x,R)$ of radius $R$ in the intrinsic Riemannian metric of the leaf can be given by $C^r$ embeddings which depend continuously on $x$ in $C^N$ topology and, if $r \notin \N$,
have $\a$-H\"older derivative of order $N$ with uniformly bounded H\"older constant.
Similarly, for such a foliation we will say that a function $g$ is {\em uniformly $C^r$ along $\w$} 
if its restrictions to $B^\w(x,R)$ depend continuously on $x$ in $C^N$ topology and
have $\a$-H\"older derivative of order $N$ with uniformly bounded H\"older constant.
We also allow $r=\infty$, in which case uniformly $\Ci$ means uniformly $C^N$ for each $N$.

\begin{definition}[Normal forms on a contracting foliation]\label{NFfol} $\;$\\
Let $f$ be a $C^1$ diffeomorphism of a compact manifold $\M$,
and let $\E$ be a continuous $f$-invariant subbundle of $T\M$ tangent to
an $f$-invariant topological foliation $\w$ with uniformly $C^1$ leaves. 
Suppose that the linear extension $F=Df |_{T\w}$ has $(\chi,\e)$-spectrum.

We say that a family $\{ \h_x \} _{x\in \M}$ of $C^1$ diffeomorphisms  
$\,\h_x: \w_x \to T_x\w$, which depend continuously on $x$ in $C^1$ topology, 
is a {\em normal form for $f$ on $\w$} if   for each $x \in \M$ we have $\h_x(0)=0$ and $D_0 \h_x =\Id \,$,
 $$
 \p_x =\h_{fx} \circ f \circ \h_x ^{-1}:T_{x}\w \to T_{f{(x)}}\w \,\text{ is in  }\s_{x,fx}, 
 $$
and for any  $y \in \w_x$ the map $\h_y \circ \h_x^{-1}:T_{x}\w \to T_{y}\w$ is
a composition of a sub-resonance polynomial in $\s_{x,y}$ with a translation.
\end{definition}

Identifying $\w(x)$ with $T_x\w$ by $\h_x$ we can view the transition maps 
$\h_y \circ \h_x^{-1}$ as maps of $T_{x}\w$ and see that they are in 
 the finite-dimensional group $\bar \s_x$ generated by  $\s_x$ and the translations of $T_x\w$.

In Theorems \ref{th weak fol} and  \ref{th weak hol} we will use normal forms for $f$ on $\w=\w^{ss}$.
In this case the leaves of $\w$ are uniformly $C^\infty$, so we can obtain a natural 
uniformly $C^\infty$ extension of $f$ by locally identifying $T_x\w$ with $\w(x)$ and apply 
Theorem \ref{NFext}. The following theorem yields a normal form for $f$ on $\w=\w^{ss}$, 
which is uniformly $C^\infty$. 
The proof that $\h_y \circ \h_x^{-1} \in \bar \s_x$ is given in \cite{KS16,KS17}.

\begin{theorem}\cite[Theorem 4.6]{K24} 
{\em (Normal forms for foliations with $C^r$ leaves)}\label{NFstrong} $\;$\\
Let $f$ be a $C^r$, $r\in (1,\infty ]$, diffeomorphism of a compact manifold $\M$,
and let $\w$ be an $f$-invariant topological foliation of $\M$ with uniformly $C^r$ leaves. 
Suppose that 
the linear extension $F=Df |_{T\w}$ has $(\chi,\e)$-spectrum with  $\chi_1/\chi_\ell <r$ and $\e<\e_0$.
 Then there exists a normal form  for $f$ on $\w$  such that 
 $\,\h_x: \w_x \to T_x\w$ are uniformly $C^r$ diffeomorphisms.
 \end{theorem}

\subsection{Normal forms on  $C^1$ leaves}
In Theorem \ref{th strong} we will use normal forms for $f$ on $\w=\w^{ws}$.
 In general, the leaves of $\w=\w^{ws}$ are only  $C^{1+\text{H\"older}}$ and
so the above result may not apply. In this case we construct the normal form 
using  the next theorem and smoothness of $\w^{ss}$ inside $\w^s$. 
The latter yields that  holonomies of $\w^{ss}$ are $C^\infty$, between 
$C^\infty$ transversals to $\w^{ss}$ inside $\w^{s}$.

\begin{theorem}\label{NF weak} {\em(Normal forms  for foliations with $C^1$ leaves)}\\
Let $f$ be a $C^r$ diffeomorphism of a  compact manifold  $\M$.
Let $U$ be an $f$-invariant topological foliation of $\M$ with uniformly $C^r$ leaves.
Let $\w$ and $\v$ be $f$-invariant topological subfoliations of $U$ with uniformly $C^1$ leaves 
 transverse in the leaves of $U$, i.e., $T_xU=T_x\w \oplus T_x\v$ 
for each $x \in \M$.  

Suppose that the holonomies $\H$ of \,$\v$ inside $U$ are uniformly $C^r$,  
and  $Df |_{TW}$ has $(\chi, \e)$ spectrum with $\chi_1/\chi_\ell <r$ and $\e<\e_0$. 
Then there is a normal form $\{\h_x\}$ for $f$ on $\w$  such that 
for any $x \in \M$ and any $y\in \v(x)$ the lifted holonomy maps  
\begin{equation}\label{hol inv eq}
\bar H_{x,y}=\h_y \circ \H_{x,y} \circ \h_x^{-1} : \,T_x\w \to T_y\w
\end{equation}
are $C^r$ diffeomorphisms, uniformly in $x$ and $y$.

\end{theorem}

\begin{proof}  To construct the normal form we need to define a $C^r$ extension  $\f$
that corresponds to the restriction of $f$ to the leaves of $\w$.  Since the leaves of $\w$ are only $C^1$, 
we use the action of $f$ transversally to $\v$ inside $U$ to define $\f$.
More precisely, for any $x\in \M$ we identify in a uniformly $C^r$ way a small ball $B_\rho(x)$ in $T_x\w$ 
with a $C^r$ submanifold of $U(x)$ tangent to $T_x\w$ at $x$. Thus we obtain a uniformly $C^r$ family 
$\t_x$ of transversals  to $\v$ inside $U$: 
\begin{equation}\label{transv}
i_x : T_x\w \supset B_\rho(x)  \to U(x), \quad i_x(B_\rho)=\t_x, \quad D_x i_x=\Id.
\end{equation}
Using these transversals we consider the following holonomy maps of $\v$:  
\begin{equation}\label{hol v}
\text{ $ \H_x: \t_x \to \w(x) $ and $\tilde \H_{fx} : f(\t_x) \to \t_{fx}$.}
\end{equation}
We note that   $\tilde \H_{fx}$ is  uniformly $C^r$ by the assumption on 
holonomies of $\v$  since $f(\t_x)$ and  $\t_{fx}$ are uniformly $C^r$ transversals, but $ \H_{x}$ is 
only uniformly $C^1$ as  $\w(x)$ is assumed only $C^1$. 
Using these maps and the identifications $i_x$ we define the  $C^r$ extension of $f$  
$$
\f_x\,=\, i_{fx}^{-1} \circ \tilde \H_{fx} \circ f|_{\t_x}  \circ i_x 
\,=\, (\H_{fx} \circ i_{fx})^{-1} \circ f|_{\w(x)} \circ (\H_x \circ i_x)  :\, T_x\w \to T_{fx}\w.
$$
In fact, the maps $\H_x \circ i_x$ give an atlas of local coordinates on the leaves of $\w$ 
with $C^r$ transition maps. This gives $\w(x)$ a structure of $C^r$ manifold, which is $C^1$
consistent with the submanifold structure. We will denote by $\ti \w(x)$ the leaf equipped with this $C^r$ structure.
With respect to it, the restriction of $f$ to each leaf, $f_x:\ti \w(x) \to \ti \w(fx)$, 
becomes uniformly $C^r$, as in these local coordinates it coincides with the  extension $\f_x$.

Now this setting becomes essentially the same as in the case of $C^r$ leaves. 
Theorem \ref{NFext} gives normal form coordinates $\hat \h_x :  T_x \w \to T_{x}\w$ 
for the extension $\f$ and hence 
\begin{equation}\label{weak hol eq}
\h_x= \hat \h_x  \circ  (\H_{x} \circ i_{x})^{-1} : \, \w(x) \to T_{x}\w
\end{equation}
is a normal form for $f$ on $\w$, moreover $\h_x \in C^r(\ti \w(x), T_{x}\w)$. The proof that $\h_y \circ \h_x^{-1} \in \bar \s_x$ is the same as in  \cite{KS16,KS17}.

For any $y \in \v(x)$, the holonomies of $\v$ between the transversals  $\tilde \H_{x,y}: \t_x \to \t_y$
 are uniformly $C^r$ and are related to the holonomies between the leaves $ \H_{x,y}: \w(x) \to \w(y)$  
 by $\tilde \H_{x,y}= \H_y^{-1} \circ  \H_{x,y} \circ \H_x$. By the construction of  $\h$
 this yields that  the lifted holonomies $\bar \hd_{x,y}$ are also uniformly $C^r$, since
\begin{equation}
\bar \hd_{x,y} =\h_y \circ \H_{x,y} \circ \h_x^{-1}
= \,\hat \h_y  \circ i_{y}^{-1} \circ \tilde\H_{x,y} \circ i_{x} \circ \h_x^{-1} : \,T_x\w  \to T_y\w.
\end{equation}

Even though the holonomies $\H_{x,y}: \w(x) \to  \w(y)$ are only $C^1$, the holonomies $\H_{x,y}: \ti \w(x) \to \ti \w(y)$ are $C^r$. 
 \end{proof}

In the setting of the previous theorem, we formulate our main technical result
on relationship between holonomies and normal forms. We will apply this theorem  
with $U=\w^s$ to both $(\v,\w)=(\w^{ss},\w^{ws})$ and $(\v,\w)=(\w^{ws},\w^{ss})$ .

\begin{theorem} [Holonomy invariance of normal forms]  \label{Main Hol} 

In addition to the assumptions of Theorem \ref{NF weak}, suppose that the foliation $U$ is  contracted by $f$,  i.e., $\| Df |_{TU}\|<1$ for some metric. Then for any $x\in \M$ and $y \in \v(x)$  the lifted holonomy  map  \eqref{hol inv eq} is a sub-resonance polynomial, i.e., $\bar H_{x,y} \in \s_{x,y}$.
   
\end{theorem}


\subsection{Proof of Theorem \ref{Main Hol}}$\;$\\
We denote $T\w$ by $\E$ and let $\E=\E^1\oplus \dots \oplus \E^\ell$ be the invariant splitting for $Df|_\E$.
\vskip.05cm

First we show that the derivative $D\H_{x,y}=D \bar H_{x,y}: \E_x \to \E_y$ preserves the flag of strong subbundles,
 i.e., for each $k=1, \dots ,\ell,$
$$D\H_{x,y} (\E_x^{1,k}) \subseteq  \E_y^{1,k}, \quad\text{where }\;\E^{1,k}=\E^1 \oplus \dots \oplus \E^k.$$
Since the foliations $\w$ and $\v$ are $f$-invariant, we have the commutation relation 
\begin{equation}\label{hol comm}
 f^{n} \circ \H_{x,y}= \H_{f^nx,f^ny} \circ f^n :\, \w_x \to \w_{f^ny}.
\end{equation}
Differentiating this relation  we obtain 
$$Df^n|_{\E_y}\circ D\H_{x,y}(x)= D \H_{f^nx,f^ny}(f^nx) \circ Df^n|_{\E_x} :\, \E_x \to \E_y.$$
We consider a non-zero vector $u\in \E^k_x$ and let $j$ be largest index so that the $j$ component of 
$v=D\H_{x,y}(x) u$ is non-zero. We need to show that $j \le k$. Suppose that $j>k$ and hence
$\xs_k+\e < \xf_j-\e$ since $\e<\e_0$. Then we obtain a contradiction as follows. First we have
$$\|Df^n|_{\E_y}\circ D\H_{x,y}(x)\, u \|=\| Df^nv \| \ge e^{(\xf_j-\e)n} \|v_j\|.
$$
On the other hand, 
$$\|D \H_{f^nx,f^ny}(f^nx) \circ Df^n|_{\E_x} u\| \,\le\, \|D \H_{f^nx,f^ny}(f^nx)\| \cdot \|Df^n|_{\E_x} u\| 
\,\le\, C e^{(\xs_k+\e)n} \|u\|,
$$
since $\|D \H_{f^nx,f^ny}(f^nx)\|$ is uniformly bounded in $n$ as $\dist_\v(f^nx,f^ny)$ decreases. This is impossible for large $n$ since $\xs_k+\e < \xf_j-\e$. Therefore $D\H_{x,y}$ preserves the flag.
\vskip.3cm

We take arbitrary $x \in \M$ and $y\in \v(x)$ and consider the lifted holonomy (suppressing the bar in this proof)
$$\hd_{x,y}=\bar H_{x,y} =\h_y \circ \H_{x,y} \circ \h_x^{-1} : \,\E_x \to \E_y.$$
We write its Taylor polynomial of degree $M=\lfloor r \rfloor$ at $0\in \E_x$ as
$$T_M(\hd_{x,y})(t)= \sum_{m=1}^M \hd^{(m)}_{x,y}(t) : \,\E_x \to \E_y,$$
where $\hd^{(m)}_{x,y} : \E_x \to \E_y$ is a homogeneous polynomial of degree $m$.
Now we show inductively that this Taylor polynomial contains only sub-resonance 
terms.
Writing $f^n$ in normal form coordinates we have 
$$     \h_{f^n x} \circ f^n|_{\w(x)} \circ \h_{x}^{-1} \,=\, \p ^n_x : \,\E_x \to \E_{f^nx},$$
where  the polynomial map $\p^n_x(t)= \sum_{m=1}^d \pd^{(m)}_x(t)$ 
 contains only sub-resonance terms and in particular has degree at most $d=\lfloor \xf_1/\xs_\ell \rfloor \le M$.  With similar notations for $y$, 
 the commutation relation \eqref{hol comm} becomes
 \begin{equation}\label{hcom}
\hd_{f^nx,f^ny} \circ  \p ^n_x \,=\, \p ^n_y \circ \hd_{x,y}\,: \;\E_x \to \E_{f^ny}.
\end{equation}  

We already proved that the first derivative $\hd^{(1)}_{x,y}=D_x H_{x,y} $ 
preserves the strong flag, which means exactly that it 
is a sub-resonance linear map. 

Inductively, we assume that $ \hd^{(m)}_{x,y}$ has only sub-resonance 
terms for all $x\in \M$, $y \in \v_x$, and $m=1,\dots ,k-1$ and show that the same holds 
for $ \hd^{(k)}_{x,y}$. We split 
$$\hd^{(k)}_{x,y}=S_{x,y}+N_{x,y}$$
 into the 
sub-resonance part and the rest. It suffices to show that $N_{x,y}= 0$ for all $y \in \w_x$ that are sufficiently close to $x$. Assuming the contrary, we fix such $x$ and $y$
with $N_{x,y}\ne 0$. We will write $N_{x}$ for $N_{x,y}$ and 
$N_{f^nx}$ for $N_{f^nx,f^ny}$. We consider the Taylor  terms of order $k$ in 
the commutation relation \eqref{hcom}. They come from the compositions of the respective
Taylor polynomials
$$ \sum_{j=1}^M \hd^{(j)}_{f^n x,f^n y}\left( \sum_{m=1}^d \pd ^{(m)}_{x}(t) \right) \; \text{ and }\;
\sum_{m=1}^d P^{(m)}_{y}\left( \sum_{j=1}^M \hd^{(j)}_{x,y}(t)  \right).
$$
By the inductive assumption, $N_{x}$  and $N_{f^nx}$ are the only non sub-resonance terms 
of order $k$ in these polynomials, and all lower order terms are sub-resonance.
Since any composition of sub-resonance  terms is again sub-resonance, 
taking non sub-resonance  terms of order $k$ on both sides yields the equation
 \begin{equation}\label{NRcom}
N_{f^n x}\left(  \pd ^{(1)}_{x}(t) \right) =  P^{(1)}_{y}\left( N_{x}(t)  \right).
\end{equation}
We decompose $N_x$ into components $N_x=(N^1_x,\dots, N^\ell_x)$ and let $i$ be the largest index such that  $N^i _x \ne 0$, i.e., there exists $t' \in \E_x$ so that $z= N^i_x (t')\in \E_y$
 has component $0\ne z_i\in \E^i_y$.
 We denote
  \begin{equation}
 w= P^{(1)}_{y}(z) =  D \p _y^n(0) z = Df^n (y) z \in \E_{f^n y} \;\,  \text{ and }\;\, 
 w_i =Df^n (y) z_i \in \E^i_{f^n y} 
 \end{equation}
its $i$th component.  Then 
 \begin{equation}\label{right}
 \| w_i \|  \ge C e^{n (\xf _i -\e)},
 \end{equation}
where the constant $C=\|z_i\|$ does not depend on $n$. 
 \vskip.1cm
 
 Now we estimate from above the $i$th component of $N_{f^n x}(  \pd ^{(1)}_{x}(t'))$. 
 First, 
 $$\| \pd ^{(1)}_{x}(t'_j) \| = \| Df^n(x) \,t'_j \|\le  \| t' \|\, e^{n  (\xs _j +\e)}
 \quad\text{for any }j .$$
Let $N^s_{f^n x}$ be a  term of homogeneity type $s=(s_1, \dots, s_\ell)$ in the component $N_{f^n x}^i$.  
By homogeneity, its norm can be estimated as 
 $$
 \| N^s_{f^n x} \left(  \pd ^{(1)}_{x}(t')\right) \| \le 
\| N_{f^n x}\| \cdot  \| t' \|^k \cdot  e^{n \sum s_j (\xs _j +\e)} .
$$
Since no term in $N_{f^n x}^i$ is a sub-resonance one, we have $\xf _i > \sum s_j \xs _j$ and hence, since  $\e<\e_0$, we also have  $  \xf _i  > \sum s_j \xs_j + (n+2)\e$.
Then the left side of \eqref{NRcom}  at $t'$ can be estimated as
$$
\| N^s_{f^n x} \left(  \pd ^{(1)}_{x}(t')\right) \| \le C'  e^{n (\xf _i -2\e)}, 
$$
where the constant $C'$ does not depend on $n$ since the norms of 
$\hd^{(k)}_{f^n x,f^n y}$, and hence those of $N_{f^n x}$, are uniformly bounded.
This contradicts  \eqref{right} for large $n$.
\vskip.1cm

Thus we have shown that $T_M(\hd_{x,y})$, the Taylor polynomial of degree $M$ 
 for $H_{x,y}$ at $0\in \E_x$,  contains only sub-resonance terms for all $x\in \M$ and $y\in \v(x)$. 
It remains to show that $H_{x,y}$ coincides with $T_M(\hd_{x,y})$.

\vskip.1cm

In addition to \eqref{hcom}, the same commutation relation holds for the Taylor polynomials
 \begin{equation}\label{Hcom}
T_M(H_{f^n x,f^n y}) \circ \p ^n_x =  \p ^n_y  \circ T_M(H_{x,y}).
\end{equation}
Indeed, the  two sides must have the same terms up to order $M$, 
but they are sub-resonance polynomials and thus have no terms of 
order higher than $d\le M$.
Denoting 
$$\Delta_n= H_{f^n x,f^n y}-T(H_{f^n x,f^n y})
$$
we obtain from \eqref{hcom} and \eqref{Hcom} that
 \begin{equation}\label{Hcom2}
 \p ^n_y  \circ H_{x,y} -  \p ^n_y  \circ T_M(H_{x,y}) = \Delta_n \circ \p ^n_x.   
\end{equation}
We denote 
$$\Delta= \hd_{x,y}-T_M(H_{x,y}) :\;\E_x \to \E_y
$$ 
and suppose 
that $\Delta \ne0$ for some $x,y$. 
Let $i$ be the largest index for which the $i$th component of $\Delta$ is 
nonzero. Then there exist arbitrarily small $t' \in \E_x$  such that the $i$th component
$z_i$ of  $z=\Delta (t') $ is nonzero. Since $\p ^n_y$ is a sub-resonance 
polynomial, the {\it nonlinear} terms in its $i$th component can depend only
on $j$ components of the input with $j>i$, which are the same for
$H_{y,x}$ and $T(H_{x,y})$ by the choice of $i$. The linear part of $\p ^n_y$ is
$Df^n(y)$ and it preserves the splitting. Thus the $i$th component of the left side of 
\eqref{Hcom2} at $t'$ is $Df^n(y) z_i$. So we can estimate  the left side of \eqref{Hcom2}  at $t'$
\begin{equation}\label{HcomR}
\| ( \p ^n_y  \circ H_{x,y} -  \p ^n_y  \circ T(H_{x,y}))(t') \| \ge 
\| Df^n(y) z_i\| \ge e^{n  (\xf _i - \e)}  \| z_i \| \ge e^{n  (\xf_1 -\e)}  \| z_i \|_ .
\end{equation}

Now we estimate  the right side of \eqref{Hcom2}.
Since $H_{f^n x,f^n y}$ is $C^{M+\alpha}$,  there exists a constant
$C$ determined by $\| H_{f^n x,f^n y}\|_{C^{M+\alpha}} $ such that 
\begin{equation}\label{Ck,delta}
\| \Delta_n (t)\| \le C \| t\| ^{M+\a} \quad\text{for all }
t \in  \E_{f^n x}  \text{ with } \| t\| \le \delta
\end{equation}
  for a sufficiently small $\delta >0$. This $C$ can be chosen uniform in $n$ and close $x,y$.
  
  Also, for a sufficiently small $\delta >0$ we  can estimate  $ \p ^n_x $   as 
\begin{equation}\label{P est}
\| \p ^n_x  (t) \| \le e^{n  (\xs _\ell +2\e)} \| t\| \quad\text{for all $n \in \N$ and $\| t\|<\delta $.} 
\end{equation}
This follows  from the fact that for all $x \in \M$ we have 
$$ \| D_0 \p _x  \| = \| Df|_{\E_x}  \| \le e^{\xs _\ell +\e}  \quad \text{and  hence } \quad  \| D_t \p _x  \| \le e^{\xs _\ell +2\e}$$
for all points $t \in \E_x$ with $\| t\| \le \delta$ for a sufficiently small $\delta >0$.
Hence $\p_x$ is a $e^{\xs _\ell +2\e}$-contraction on the $\delta$ ball in each $\E_x$,
and \eqref{P est} follows.

 Combining  \eqref{Ck,delta} and \eqref{P est} we estimate  the right side of \eqref{Hcom2} at $t'$ as
$$
\|\left( \Delta_n \circ \p ^n_x \right) (t') \|\,\le \, C \|  \p ^n_x  (t')\| ^{M+\a}
\le C  \| t'\| ^{M+\a} e^{n (M+\a) (\xs _\ell +2\e)}.
$$
Now we see that this contradicts  \eqref{HcomR} for
large $n$ since $(M+\a) \xs_{\ell}= r \xs_{\ell}<\xf _1$ and since  $\e<\e_0$ is sufficiently small.
Thus, $\Delta=0$, i.e., the holonomy map $H_{x,y}$ coincides with its Taylor polynomial of order $M$.
This completes the proof of Theorem \ref{Main Hol}.


 \section{Proof of Theorem \ref{th strong}} \label{hsmoothWs}
 
We first note that\\ 
(3)$\implies$(2)\,  is clear, and \\
(1)$\implies$(2)\, holds since we  always have $h(\w^{u})=W^{u}$

\hskip1.8cm  and hence the joint foliation  for $\w^{u}$ and $\w^{ss}$ is $h^{-1}(W^u\oplus W^{ss})$.\\ 
 The converse (2)$\implies$(1) is part of \cite[Theorem 1.1]{GSh}. 

Next we prove that \\
(1)$\implies$(5)$\implies$(4) and\\
(1)$\implies$(3), \,more precisely (5)+(2)$\implies$(3).

Finally we show (4)$\implies$(1) under the assumption on density of Lyapunov leaves.
 \vskip.2cm

{\bf (1)$\implies$(5).}
We assume that $h(\w^{ss})=W^{ss}$. We apply Theorems \ref{NF weak} and \ref{Main Hol}
with $U=\w^s$, $\w=\w^{ws}$ and $\v=\w^{ss}$, and  denote $\E^{ws}=T\w^{ws}$. 
Since $f$ is $C^1$ close to $L$, the linear extension 
$Df|_{\E^{ws}}$ has the corresponding $(\chi,\e)$ spectrum with small $\e$. 
Since $\w^{ss}$ is a $C^\infty$ subfoliation inside the leaves of $\w^s$, Theorems \ref{NF weak} 
 applies and yields the normal form coordinates  $\h_x$ on $\w^{ws}$  given by equation \eqref{weak hol eq}
$$\h_x:\,  \w^{ws}(x) \to \E_x^{ws}. \quad 
$$
The maps $\h_x$ are $C^1$ diffeomorphisms which depend continuously on $x$ in $C^1$ topology.
Further, Theorem \ref{Main Hol} yields that the holonomies 
$\H=\H^{ss}$ of $\w^{ss}$ inside $\w^{s}$ between leaves of $\w^{ws}$
are sub-resonance polynomials  in these coordinates 
\eqref{hol inv eq}, i.e., 
$$ \bar H_{x,y}=\h_y \circ \H_{x,y} \circ \h_x^{-1} :\, \E_x^{ws} \to \E_y^{ws}\quad\text{is in }\s_{x,y}.$$ 
Since $h(\w^{ss})=W^{ss}$ and $h(\w^{ws})=W^{ws}$, the two foliations form a global product structure inside the leaves of $\w^s$ conjugate by $h$ to that of the linear foliations. In particular,
the holonomies $\H$ are globally defined on the leaves of $\w^{ws}$. The corresponding 
linear holonomies $H$ for $L$ are translations along $W^{ss}$: 
$$\text{if $y\in W^{ss}(x)$\, then $\,H_{x,y}(z)=z+ (y-x)$,}$$
 and 
  $h$ conjugates $\H$ and $H$ as follows
 $$
 H_{h(x),\,h(y)}=h\circ \H_{x,y} \circ h^{-1}.
 $$

We fix an arbitrary $x\in \T^d$ and $y\in \w^{ws}(x)$. By the assumption, the linear leaf $W^{ss}(x)$ is dense in $\T^d$. Hence there exists a 
sequence of vectors $v_n\in E^{ss}$ such that $h(x)+v_n$ converges to $h(y)$. Denoting 
$y_n=h^{-1}(h(x)+v_n)$ we obtain a sequence of points $y_n\in \w^{ss}(x)$ converging to $y$.

The corresponding linear holonomies $H_{v_n}=H_{h(x),\,h(y_n)}$
converge in $C^0$ to the translation $H_v$ in $W^{ws}(h(x))$ by the vector $v=h(y)-h(x)$. Hence
the holonomies $\H_{x,y_n}$ converge in $C^0$ norm to the homeomorphism 
$$\H_{v}: \w^{ws}(x) \to \w^{ws}(y)= \w^{ws}(x) \;\text{ such that }H_{v}=h\circ \H_{v} \circ h^{-1}.
$$
Since the normal form coordinates $\h_y$ depend continuously on $y$, the corresponding lifted holonomies 
$ \bar H_{x,y_n}=\h_{y_n} \circ \H_{x,y_n} \circ \h_x^{-1} \in \s_{x,y_n}$ 
converge to the homeomorphism  
 $$  P_{x,y}=  \h_{y} \circ  \H_{v} \circ  \h_x^{-1} =
  \h_{y} \circ  h^{-1}\circ H_{v} \circ h \circ  \h_x^{-1}  :\;  \E_x^{ws} \to \E_y^{ws}.
 $$
The map $P_{x,y}$ is also a sub-resonance polynomial, i.e., $ P_{x,y}\in \s_{x,y}$.
  
Now we lift  the restriction of $h$ to $\w ^{ws}(x)$ to $\E_x$  using coordinates $\h_x$
  $$
  \bar h_x = h \circ \h_x^{-1} :\; \E_x^{ws} \to W^{ws}(h(x)),
  $$
and  conjugate  by $\bar h_x$ the translation $H_v$ of $W^{ws}(h(x))$ 
to the corresponding map of $\E_x$
$$ 
\bar H_v=   (\bar h_x)^{-1}\circ H_{v} \circ \bar h_x =
  \h_{x} \circ  h^{-1}\circ H_{v} \circ h \circ  \h_x^{-1}=  \h_{x} \circ \h_y^{-1} \circ P_{x,y} :\;  \E_x^{ws} \to \E_x^{ws}.
 $$
By the property of normal form coordinates, the map $\h_{x} \circ \h_y^{-1} :\;  \E_y^{ws} \to \E_x^{ws}$
is a composition of a sub-resonance polynomial in $ \s_{y,x}$ with a translation of $\E_x^{ws}$.
Since  $P_{x,y}$ is a  sub-resonance polynomial in $ \s_{x,y}$, we conclude that 
$\bar H_v \in \bar \s_x$, the finite dimensional Lie group  generated by $\s_x$ and the translations of $\E_x$.
  Thus $\bar h_x$ conjugates the action of $E^{ws}=\R^k$ by translations of $W^{ws}(h(x))$ 
with the corresponding continuous action of $\R^k$ by elements of the Lie group $\bar \s_x$.
This yields the injective  continuous homomorphism 
 $$
 \eta_x : E^{ws} \to  \bar \s_x \quad\text{given by }\; 
 \eta_x (v)= \bar H_v =(\bar h_x)^{-1} \circ H_v \circ \bar h_x.
 $$ 
It is a classical result that a continuous homomorphism between Lie groups is automatically a
 Lie group homomorphism, and so it is automatically $C^\infty$, see
for example \cite[Corollary 3.50]{Ha}. Thus $\eta_x$ is  a $C^\infty$ diffeomorphism onto its image in $\bar \s_x$.  
We conclude that $(\bar h_x)^{-1}$ is a $C^\infty$ diffeomorphism between $W^{ws}(h(x))$ and 
  $\E^{ws}_x$ since it is  determined by  $\eta_x$ as follows 
 $$
( \bar h_x)^{-1}(h(x)+v)= \bar H_v (0)=\eta_x (v)(0).
 $$

 Hence $\bar h_x  : \E_x^{ws} \to W^{ws}(h(x))$ is also a $C^\infty$ diffeomorphism.
 Further, since the normal form coordinates $\h_x$, as well as holonomies and their limits, 
 depend continuously on $x$, the constructed continuous action on $\E^{ws}_x$ and 
 the corresponding homomorphism $\eta_x$ also depend continuously on $x$. 
 This implies that $\eta_x$ depend continuously on $x$ in $C^\infty$ topology,
because it is determined by the corresponding linear homomorphism 
 of the Lie algebras. This yields that 
$\bar h_x  :\; \E_x^{ws} \to W^{ws}(h(x))$
 also depends continuously on $x$ in $C^\infty$ topology.
 We conclude that the restriction of $h$ to $\w ^{ws}(x)$
 $$
 h|_{\w^{ws}(x)} =\bar h_x \circ \h_x  : \; \w^{ws}(x) \to W^{ws}(h(x))
 $$
  is as regular as $\h_x$. 
If the leaves of $\w^{ws}$ are uniformly $C^{q}$, then so are the maps $\h_x$ given by
   \eqref{weak hol eq} since the holonomy $\H_x$ from a smooth transversal 
   to the leaf $\w^{ws}(x)$ is as regular as the leaves. Hence in this case $h|_{\w^{ws}(x)}$ is
  a uniformly $C^{q}$  diffeomorphism.    In particular, since the leaves  of $\w^{ws}$ are at least uniformly $C^{1+\text{H\"older}}$,  $h|_{\w^{ws}(x)}$ is at least a uniformly $C^{1+\text{H\"older}}$  diffeomorphism.
 
 \vskip.1cm 

To complete the proof of (5) we will now show that the component $h^{ws}$ is $C^\infty$ on $\T^d$. 
First, $h^{ws}$ is (locally) constant along $\w^{u+ss}$ and hence is uniformly
 $C^\infty$ along it, since its leaves are uniformly $C^\infty$. By Journ\'e lemma 
 \cite{J},\cite[Theorem 3.3.1]{KtN} it suffices to show that $h^{ws}$ is also uniformly 
 $C^\infty$  along smooth transversals to $\w^{u+ss}$. 
 
 As in the construction of normal forms on $\w^{ws}$ we consider the  embeddings 
  of small balls $B^{ws}_\rho(x)$ in $\E^{ws}_x$
as a local family of smooth transversals \eqref{transv}
 and denote
$$i_x :  B^{ws}_\rho(x)  \to \w^{s}, \quad i_x(B^{ws}_\rho(x))=\t_x, \quad  \H_x= \H_x^{\w^{ss}}: \t_x \to \w^{ws}(x).
$$
We recall that the normal form coordinates  \eqref{weak hol eq} are given by  
 $$
 \h_x= \hat \h_x  \circ  (\H_{x} \circ i_{x})^{-1} : \; \w^{ws} \to \E^{ws}_{x},
$$
where $ \hat \h_x :\E^{ws}_{x} \to \E^{ws}_{x}$ is the uniformly $C^\infty$ coordinate change 
for the extension $\f_x$.
Since we know that both $\bar h_x$  and $i_{x} \circ (\hat \h_x  )^{-1} $ are uniformly $C^\infty$ diffeomorphism and since
 $$
 \bar h_x = h \circ (\h_x)^{-1}=  h \circ \H_{x} \circ i_{x} \circ (\hat \h_x  )^{-1}: \; \E_x^{ws} \to W^{ws}(h{x})
$$
we conclude that  
 $ h \circ \H_{x} : \; \t_x \to W^{ws}(h{x})
 $
 is also uniformly $C^\infty$,  and in particular so is its $ws$-component. 
We claim that $(h \circ \H_{x})^{ws}=(h|_{\t_x})^{ws}$.   Indeed,  for any $y \in \t_x$ we have $\H_{x}(y) \in \w^{ss}(y)$ and, since $h(\w^{ss})=W^{ss}$,  we get $h (\H_{x}(y)) \in W^{ss}(h(y))$ and thus they have  the same $ws$-component. 
Therefore,  $(h|_{\t_x})^{ws}=h^{ws}|_{\t_x}$  is uniformly $C^\infty$ 
along this family of transversals to $\w^{u+ss}$, and
 we conclude that  $h^{ws}$ is $C^\infty$ on $\T^d$.

 This concludes the proof of\, (1)$\implies$(5).
  \vskip.3cm
  
{\bf (5)$\implies$(4).} Since (5) yields that $h|_{\w^{ws}(x)}$ is  a  $C^{1}$  diffeomorphism,  it remains to show that its derivative 
 is H\"older on $\T^d$.  Since $h$ maps ${\w^{ws}(x)}$ to the linear leaf $W^{ws}(h(x))$, this derivative coincides with the restriction of the derivative of $h^{ws}$,
 $$D(h|_{\w^{ws}(x)})(x)=D(h^{ws}|_{\w^{ws}(x)})(x)= D(h^{ws})|_{\E^{ws}(x)}:\E^{ws}_x \to E^{ws} \subset \R^d.$$
 Since $h^{ws}$ is $C^\infty$ on $\T^d$ by (5), this restriction is as regular as the subbundle $\E^{ws}_x$,
and so at least H\"older continuous on $\T^d$.
 \vskip.3cm
 
{\bf (1)$\implies$(3).} Since (1) implies (2), we have the joint topological 
foliation $\w^{u+ss}=h^{-1}(W^{u} \oplus W^{ss})$. Now we use (5), which follows from (1), 
to show that $\w^{u+ss}$ is conjugate to the linear foliation $W^{u} \oplus W^{ss}$ by  a $C^\infty$ diffeomorphism. 
We take $h$ to be the conjugacy close to the identity and write 
$$h(x)=x+\d(x)=x+\d^{ws}(x)+\d^{u+ss}(x),
$$
 where $\d : \T^d \to \R^d$ is split into components $\d^{ws} : \T^d \to E^{ws}$ and  $\d^{u+ss} : \T^d \to E^u\oplus E^{ss}$.
Now we consider the  map
$$\tilde h (x)=h(x)-\d^{u+ss}(x)=x+\d^{ws}(x): \,\T^d \to \T^d.
$$
We note that both $h$ and $\tilde h$ are $C^0$ close to the identity.
The first formula shows that $\tilde h$ is an adjustment of $h$ along $W^{u} \oplus W^{ss}$
and thus, since $h(\w^{u+ss})=W^{u} \oplus W^{ss}$, we also  have that $\tilde h(\w^{u+ss})=W^{u} \oplus W^{ss}$. Now we show that $\tilde h$ is a $C^\infty$ diffeomorphism of $\T^d$, and hence it smoothly 
conjugates  $\w^{u+ss}$  to $W^{u} \oplus W^{ss}$. 
We can locally write 
$$\tilde h (x)=x+\d^{ws}(x)=x^{u+ss}+x^{ws}+\d^{ws}(x)=x^{u+ss}+ h^{ws}(x).$$ 
 While the components $h^{ws}$, $x^{u+ss}$, and $x^{ws}$ are globally well-defined 
 only for the lifts to $\R^d$, they make sense locally on $\T^d$.
By (5), the component $h^{ws}$ is $C^\infty$ on $\T^d$ and hence so is $\tilde h $. 
Thus to prove that $\tilde h$ is  locally a $C^\infty$ diffeomorphism it suffices to show 
invertibility of its derivative. We view $D\tilde h$ with respect to splittings 
$\E^{ws}_x \oplus \E^{u+ss}_x \to E^{ws} \oplus E^{u+ss}$. Since $h^{ws}$ is constant
along $\w^{u+ss}$ and a $C^1$ diffeomorphism along $\w^{ws}$ we see that $D\tilde h$ is 
block triangular with invertible block $\E^{ws}_x  \to E^{ws}$. Now invertibility of  $D\tilde h$
follows from that of the other diagonal block  $D(x^{u+ss})|_{\E^{u+ss}_x}$. The latter 
 is clear since $\E^{u+ss}_x$ is transverse to $E^{ws}$ and thus the derivative of
$x^{u+ss}$, which is just $E^{u+ss}$-component of the identity, has full rank on $\E^{u+ss}_x$.
Thus $\tilde h$ is  locally a $C^\infty$ diffeomorphism.
The global surjectivity and injectivity of $\tilde h$ also follow from this since it is $C^0$ close to the identity.
 This concludes the proof of\, $(1)\implies (3)$.

\vskip.3cm

{\bf (4)$\implies$(1)\,  if the leaves of $W^1, \dots, W^\ell$ are dense in $\T^d$.}

Since $f$ is $C^1$ close to $L$, the linear extension $Df|_{\E^{ws}}$ has the corresponding $(\chi,\e)$ spectrum and the  H\"older continuous splitting
\begin{equation} \label{splitEws} \E^{ws}=\E^{1} \oplus \dots \oplus \E^{\ell}
\end{equation} 
which is $C^0$ close to the Lyapunov splitting for $L$ on $E^{ws}$. 
We recall that $\E^1$ is the strongest contraction and $\E^\ell$ is the weakest. We let
$$\E^{i,ss}=\E^{i} \oplus \dots \oplus \E^{1} \oplus \E^{ss}.$$
Since $\E^{i,ss}$ is a strong subfbundle of $\E^s$, it is tangent to a  $C^\infty$ subfoliation $\w^{i,ss}$
inside $\w^s$.
Since we assume that $h|_{\w^{ws}(x)}$ is a $C^{1}$  diffeomorphism, each bundle $\E^i$ is tangent to 
the foliation $\w^i=h^{-1}(W^i)$.  
The implication (4)$\implies$(1) follows from the next proposition.

\begin{proposition} \label{prop_for_ind}
Assume that $h$ is  a $C^{1}$ diffeomorphism along $\w^{ws}$ with the derivative 
$$D(h|_{\w^{ws}(x)})(x):\,\E^{ws}_x \to E^{ws} \subset \R^d \,\,\text{ H\"older continuous on $\T^d$.}
$$ 
If $\,h(\w^{i,ss})=W^{i,ss}$ then $\,h(\w^{i-1,ss})=W^{i-1,ss}$.
\end{proposition}

We apply the proposition inductively from $i=\ell$ to $i=1$. 
For $i=\ell$, the assumption $h(\w^{i,ss})=W^{i,ss}$ is satisfied since it is $h(\w^s)=W^s$. For $i=1$ we obtain (1).

The proof of this proposition is similar to that of \cite[Proposition 2.5]{GKS11}. However, in \cite{GKS11} the automorphism $L$ was assumed to be irreducible, while we only assume density of the leaves of $W^{ss}$ and the Lyapunov subfoliations of $W^{ws}$.

\vskip.1cm
\noindent {\it Proof of Proposition \ref{prop_for_ind}.}  We will use the
following notation in this proof:
$$
\VL=W^i, \quad \UL=W^{i-1,ss}, \quad \WL=W\oplus V = W^i \oplus W^{i-1,ss} =W^{i,ss}, 
$$
and similarly for the corresponding nonlinear invariant foliations of $f$,
$$
\Vf=\w^i, \quad \Uf=\w^{i-1,ss}, \quad \Wf=\w\oplus \w = \w^{i,ss}. 
$$

By the assumption, $h(\Wf)=\WL$.
We let $\F=h(\Uf)$. Then $\F$ is a subfoliation of $\WL$ with
continuous leaves.
We need to show that $\F=\UL$. Since $\VL=h(\Vf)$, the foliation $\F$ is
topologically transverse to $\VL$, i.e., any  leaf of $\F$
and any leaf of $\VL$ in the same leaf of $\WL$ intersect at
exactly one point.
Thus for any point $a\in\Td$ and any $b\in\F(a)$ we can define the holonomy map $\ti H_{a,b} :
\VL(a) \to \VL(b)$ along the  foliation $\F$.
The key step is to show that $\ti H_{a,b}$ a parallel translation inside $\WL$.
This is similar to \cite[Lemma 2.6]{GKS11}, but in \cite{GKS11} the automorphism $L$ was assumed to be irreducible, which yields conformality of $L$ on $\VL$. We modify the argument for the general case when
Jordan blocks may create nonconformality.

\begin{lemma}\label{translation}
For any point $a\in\Td$ and any $b\in\F(a)$ the holonomy map
$\ti H_{a,b}$  is a restriction to $\VL(a)$ of a parallel translation inside $\WL$.
\end{lemma}

\begin{proof}

For any point $c\in\Td$ and any $d\in\Uf(c)$ we denote by $\H_{c,d} : \Vf(c) \to \Vf(d)$ 
the holonomy along the foliation $\Uf$. Since $\Uf$ is a strong subfoliation of $\Wf$, it is $\Ci$ inside 
the leaves of $\Wf$ and hence the holonomies $\H_{c,d}$ are $C^{1}$ with the derivative 
$$D_c\H_{c,d}: T_c  \Vf \to T_d \Vf$$
  depending H\"older continuously on $c$ and $d$.
Since $\F=h(\Uf)$ and $h(\Vf)=\VL$ we have
$$
\ti H_{a,b} = h \circ \H_{h^{-1}(a),h^{-1}(b)} \circ h^{-1}.
$$
It follows from the  regularity assumption on $h|_{\w^{ws} }$ that  the maps $\ti H_{a,b}$ are also $C^{1}$. To show that $\ti H_{a,b}$ 
is a parallel translation, we prove that the differential $D\ti H_{a,b}=\Id$.
We apply  $L^{n}$, which  contracts $\WL$,  and denote $a_n=L^{n}(a)$ and $b_n=L^{n}(b)$. Since $\F=h(\Uf)$  and
$f$ preserves the foliation $\Uf$, the map
$L$ preserves  $\F$ and we can write
$$
\ti H_{a,b} = L^{-n}\circ  \ti H_{a_n,\,b_n} \circ L^{n}.
$$
Differentiating and denoting  $D_{a_n}\ti H_{a_n,b_n}=\Id+\Delta_n$
we obtain
$$
D_a \ti H_{a,b} = (L^{-n}|_{\VL} )\circ ( D_{a_n}\ti H_{a_n,b_n}) \circ L^{n}|_{\VL}
=  \Id + L^{-n}|_{\VL}\circ \Delta_n \circ L^{n}|_{\VL}.
$$
Since $\VL$ is a Lyapunov foliation for $L$, all eigenvalues of $L$ on $\VL$ 
have the same modulus and hence the quasiconformal distortion of $L^n|_{\VL}$
grows at most polynomially, 
$$
   \|L^{-n}|_{\VL}\| \cdot\| L^n|_{\VL}\| \le C n^{2k}
   \;\text{ for all }n,
$$
where $k+1$ is the largest size of Jordan blocks of $L^n|_{\VL}$. Thus we obtain
$$
   \|L^{-n}|_{\VL}\circ \Delta_n \circ L^{n}|_{\VL}\|\le C n^{2k} \|\Delta_n\|
   \;\text{ for all }n.
$$
It remains to show that $\|\Delta_n\| \to 0 $ exponentially in $n$. We
differentiate the equation
$
  \ti H_{a_n,b_n} = h \circ \H_{h^{-1}(a_n),\,h^{-1}(b_n)} \circ h^{-1}
$  
at $a_n$
$$
  D_{a_n}\ti H_{a_n,b_n} = (D_{h^{-1}(b_n)}h|_{\Vf})\circ
 ( D_{h^{-1}(a_n)}\H_{h^{-1}(a_n),\,h^{-1}(b_n)}) \circ (D_{a_n} (h|_{\Vf})^{-1}),
$$
and denoting  $ D_{h^{-1}(a_n)}\H_{h^{-1}(a_n),\,h^{-1}(b_n)}=\Id+\Delta_n'$\, we obtain 
$$
  D_{a_n}\ti H_{a_n,b_n} = 
 ( D_{h^{-1}(b_n)}h|_{\Vf})\circ D_{a_n} (h|_{\Vf})^{-1} +  (D_{h^{-1}(b_n)}h|_{\Vf})\circ
 \Delta_n' \circ D_{a_n} (h|_{\Vf})^{-1}.
 $$
Denoting $D_{h^{-1}(b_n)}h|_{\Vf}\circ D_{a_n} (h|_{\Vf})^{-1} =\Id+\Delta_n''$\, we conclude that
$$
\| \Delta_n \| \,=\,\|\Id-D_{a_n}\ti H_{a_n,b_n}\|\,\le\, \| \Delta_n'' \| +  \|D_{h^{-1}(b_n)}h|_{\Vf}\| \cdot 
\| \Delta_n' \| \cdot \| D_{a_n} (h|_{\Vf})^{-1}\|.
 $$
By the  regularity assumption on $h|_{\w^{ws} }$  and H\"older continuity of $T_x\Vf=\E^i_x$
we have H\"older dependence of $D_x(h|_{\Vf}):\,T_x\Vf \to \VL$ on $x$.
It follows that  that 
$$ \begin{aligned} & \|D_{h^{-1}(b_n)}h|_{\Vf}\| \cdot  \| D_{a_n} (h|_{\Vf})^{-1}\|\;\text{ is uniformly bounded 
and} \\
&\|\Delta _n''\|\,\text{ is H\"older in }\dist(a_n, b_n).
\end{aligned}
$$
 Also,   $\| \Delta_n' \| $ is H\"older in $\dist(a_n, b_n)$
since $D_c\H_{c,d}$  depends H\"older continuously on $c$ and $d$ and $D_c\H_{c,d}=\Id$.
Now since  $\dist(a_n, b_n)\to 0$ exponentially as $n\to \infty$ we conclude that so does $\|\Delta_n\| $  and hence $D\ti H_{a,b}=\Id$.
\end{proof}

Now we complete the proof of Proposition  \ref{prop_for_ind} as in \cite[Proposition 2.5]{GKS11}. 
Let $a$ be a fixed point of $L$ and  let $B$ be the unit ball
in $\UL(a)$ centered at~$a$. If $B\subset\F(a)$, then $\UL(a)=\F(a)$
by invariance of $V$ and $\tilde V$ under $L^{-1}$. By the assumption, the leaf $\UL(a)$  is dense in $\T^d$.
It follows that the set of points $x$ such that
$\UL(x)=\F(x)$ is dense in $\T^d$ and hence $\UL=\F$.
Therefore, it suffices to show that $B\subset\F(a)$.

We argue by contradiction. Assume that there is $z_1\in B$ such that
$z_1\notin\F(a)$. Let $x_1=\VL(z_1)\cap\F(a).$
Since $\VL$ has dense leaves we can choose a sequence
$\{b_n,n\ge1\}$ in $\VL(a)$ so that $b_n\to x_1$ as $n\to\infty$. Let
$y_n=\ti H_{a,x_1}(b_n).$
Continuity of $\F$ implies that the sequence $\{y_n\}$ converges to a point
$x_2  \in \F(a)$. Moreover, Lemma~\ref{translation} implies that $\{x_1,x_2\}$
is a parallel  translation of $\{a,x_1\}$.

We continue this procedure inductively to construct the sequence
$\{x_n,n\ge 1\}$ in $\F(a)$. Let $z_n=\VL(x_n)\cap\UL(a).$
Then by the construction
$$
 d_{\UL}(z_n,a)=n\cdot d_{\UL}(z_1,a) \quad\text{and}\quad
d_{\VL}(x_n,z_n)=n\cdot d_{\VL}(x_1,z_1),
$$
but this is impossible since $L$  contracts $\UL$ exponentially
stronger than $\VL$. Indeed, if we take $N(n)$ to be the smallest 
integer such that $L^{N(n)}(z_n)\in B$, then 
$$
d_{\VL}(L^{N(n)}(x_n), L^{N(n)}(z_n))\to\infty \quad \text{ as } \; n\to\infty,
$$
which contradicts $\,\max_{z\in B}\,d_{\VL}(z,\VL(z)\cap\F(a))<\infty.$\,
Thus $\F=\UL$. 
$\QED$


\section{Proof of Theorem \ref{global strong} } \label{proof global strong}
The proof of Theorem \ref{th strong} works for the global setting with only minor adjustments.
In this case, (2)$\implies$(1) as well as $h(\w^{ws})=W^{ws}$ is given by \cite[Theorem 1.1]{GSh}.
Further, \cite[Theorem 1.2]{GSh} yields the dominated splitting and spectrum for $Df|_{\E^{ws}}$ 
matching that of $L$, which implies that $Df|_{\E^{ws}}$ has $(\chi,\e)$ spectrum for any $\e>0$.  
The remaining arguments work without change,  except in (1)$\implies$(3) we only obtain
that $\tilde h$ is a local $C^\infty$ diffeomorphism; the injectivity of $\tilde h $ is not clear 
since it is not close to the identity. Still $\tilde h$ gives local foliation charts for $\w^{u+ss}$. Moreover, the H\"older continuous metric on $\E^{ws}$
invariant under the holonomies of  $\w^{u+ss}$ can be pulled by $\tilde h$ using the linear foliation
as in the proof of Theorem \ref{strong+weak}(3)$\implies$(2) below, since this is a local property.


\section{Proof of Theorem \ref{th weak fol}} \label{proof weak fol}

Theorem \ref{th weak fol} follows from Theorem \ref{th weak hol} below, where we assume only the regularity of holonomies of $\w^{ws}$ between the leaves of $\w^{ss}$.

\begin{theorem} [Rigidity of weak holonomies] \label{th weak hol} 
Let $L$ be a hyperbolic automorphism of $\,\T^d$ with dense leaves of $W^{ws}$. 
Let $f$ be a $C^\infty$ diffeomorphism sufficiently $C^1$ close to $L$, and let $h$ be a topological conjugacy between $f$ and $L$. 

Let $r_{ss}(L)$ be given by \eqref{rss} and let $q>1$ be a {\em noninteger} such that the leaves of $\w^{ws}$ are uniformly $C^q$. Then  for the statements below we have \,\, 
$$(1) \iff (1') \iff (2) \implies (3)  \implies (4).
$$ \vskip.1cm
 
\begin{itemize} 
 \item[(1)]\, Holonomies of $\w^{ws}$ between the leaves of $\w^{ss}$ are uniformly $C^r$ with $r>r_{ss}(L)$,
  \vskip.1cm
 
 \item[(1$'$)]\, Holonomies of $\w^{ws}$ between the leaves of $\w^{ss}$ are uniformly $C^\infty$,
  \vskip.1cm

 \item[(2)]\, $h^{ss}$  is a uniformly $C^\infty$ diffeomorphism along $\w^{ss}$,
  and $h^{ss}$ is $ C^{q}$ on $\T^d$, 
 \vskip.1cm
 
 \item[(3)]\, The joint foliation  $\w^{u+ws}$ is conjugate to the linear foliation $W^{u} \oplus W^{ws}$ 
 \,by a $C^{q}$ diffeomorphism, 
 \vskip.1cm

 \item[(4)]\, $\w^{ws}$ is a uniformly $C^{q}$ subfoliation of $\w^s$.
 \vskip.2cm
 
\end{itemize}  
\noindent If in addition $h(\w^{ss})=W^{ss}$, \, then\, 
 \vskip.1cm
\em{(1,\,1$'$,\,2)} $\iff$  $h$  is a uniformly $C^{\infty}$ diffeomorphism along $\w^{ss}$

\hskip1.55cm $\iff$  $h$  is a uniformly $C^r$  diffeomorphism along $\w^{ss}$ with $r>r_{ss}(L)$.

\end{theorem}

\subsection{Deducing Theorem \ref{th weak fol} from Theorem \ref{th weak hol}.}$\;$

 We recall that, for a {\em noninteger} $r>1$, a foliation is $C^r$ if and only if its leaves and local holonomy maps
are uniformly $C^r$, see e.g. \cite[Theorem 6.1(i)]{PSW}.

Also we always have  $h(\w^{ws})=W^{ws}$, and hence  $\w^{u}$ and  $\w^{ws}$ are jointly integrable and the foliation $\w^{u+ws}$ is $h^{-1}(W^{u}\oplus W^{ws})$. 
  \vskip.1cm
{\bf (1)$\implies$(2).}\, (1) implies Theorem \ref{th weak hol}(1) and hence Theorem \ref{th weak hol}(2). 
Since (1) also implies that we can take $q=r$ in Theorem \ref{th weak hol}, we  obtain that $h^{ss}$ is $ C^{r}$ on $\T^d$.
  \vskip.1cm
  {\bf (2)$\implies$(3).}\, (2) implies Theorem \ref{th weak hol}(2) with $q=r$ and hence it yields 
  Theorem \ref{th weak hol}(3)  with $q=r$, which is (3).
  \vskip.1cm
{\bf (3)$\implies$(1).}\, (3) implies Theorem \ref{th weak hol}(3) with $q=r$ and hence it yields 
  Theorem \ref{th weak hol}(4)  with $q=r$, which is (1).
  
  The additional statement for $h(\w^{ss})=W^{ss}$ also follows from the corresponding part of Theorem \ref{th weak hol}.

\subsection{Proof of Theorem \ref{th weak hol}.}
It follows the same scheme as the proof of  Theorem~\ref{th strong}. 
\vskip.05cm 
We  denote $\E^{ss}=T\w^{ss}$. Since $f$ is $C^1$ close to $L$, the linear extension  $Df|_{\E^{ss}}$ has $(\chi,\e)$ spectrum with small $\e$. Since $\w^{ss}$ has uniformly $C^\infty$ leaves, we can apply Theorems \ref{NFstrong} to 
$\w=\w^{ss}$ and obtain the normal form coordinates  $\h_x$ on $\w^{ss}$  
$$\h_x:\,  \w^{ss}(x) \to \E_x^{ss}. \quad 
$$
The maps $\h_x$  are $C^\infty$ diffeomorphisms and depend continuously on $x$ in $C^\infty$ topology.

 \vskip.3cm
{\bf (1)$\implies$(1$'$).}\,
Suppose that  holonomies $\H=\H^{ws}$ of $\w^{ws}$  inside $\w^{s}$ between the leaves of $\w^{ss}$ are uniformly $C^r$ with $r>r_{ss}(L)$.
Hence we can apply Theorem \ref{Main Hol} with $U=\w^s$, $\w=\w^{ss}$ and $\v=\w^{ws}$. 
Indeed, the lifted holonomies $\bar H$ in \eqref{hol inv eq} are also $C^r$, and so the theorem 
yields that they are sub-resonance polynomials:
\begin{equation} \label{ss hol SR}
 \bar H_{x,y}=\h_y \circ \H_{x,y} \circ \h_x^{-1} :\, \E_x^{ss} \to \E_y^{ss}\quad\text{are in }\s_{x,y},
 \end{equation}
and in particular are $C^\infty$. Since the  coordinates $\h_x$ are also uniformly $C^\infty$, we conclude that  
 holonomies $\H=\H^{ws}$  are uniformly $C^\infty$. 

\vskip.1 cm
{\bf (1)$\implies$(2).}\,
Now we use \eqref{ss hol SR} to show that $h^{ss}$ is uniformly $C^\infty$ along $\w^{ss}$.
Since we do not assume $h(\w^{ss})= W^{ss}$, we need to adjust the holonomy argument accordingly.
We fix a point $x\in \T^d$ and consider the map 
\begin{equation} \label{h hat}
\begin{aligned}
&\hat h_x :\w^{ss}(x) \to W^{ss}(h(x))\quad\text{given by}\quad 
\hat h_x = H^{ws}_{h(x)} \circ h|_{\w^{ss}(x)},\;\text{ where}\\
&\text{$H^{ws}_{h(x)}:h(\w^{ss}(x))\to W^{ss}(h(x))$ is the linear holonomy
along $W^{ws}$.}
\end{aligned}
\end{equation}
We will prove that the maps $\hat h_x$ are uniformly $C^\infty$. 
\vskip.1cm

We fix $y\in \w^{ss}(x)$ and take a sequence of points $y_n\in \w^{ws}(x)$ converging to $y$.
This can be done since the leaves of the linear foliation $W^{ws}$ are dense in $\T^d$ and 
 the leaf conjugacy $h$ is a homeomorphism which sends $\w^{ws}$ to $W^{ws}$.
 
 Since $h(\w^{ws})= W^{ws}$, the holonomy maps $\H_{x,y_n}: \w^{ss}(x) \to \w^{ss}(y_n)$
 are conjugated to the corresponding linear  holonomies $H_{h(x),\,h(y_n)}:W^{ss}(h(x)) \to W^{ss}(h(y_n)),$
$$
\H_{x,y_n}=(\hat h_{y_n})^{-1}\circ  H_{h(x),\,h(y_n)}\circ  \hat h_x : \w^{ss}(x) \to \w^{ss}(y_n).
$$ 
We note that $H_{h(x),\,h(y_n)}$ are  translations $H_{v_n}$ by the vectors $v_n=h(y_n)-h(x)$.
Since $y_n$ converge to $y$, and hence $h(y_n)$ converge to $h(y)$, we see that for $v=h(y)-h(x)$, 
$$
 H_{h(x),\,h(y_n)} =H_{v_n} \;\text{ converge to }\;
H_{v} : W^{ss}(h(x)) \to W^{ss}(h(y)).
$$ 
Since $\hat h_y$ depend continuously on $y$, we obtain $C^0$ convergence 
$$
 \H_{x,y_n} \,\text{ converge to }\; 
(\hat h_{y})^{-1}\circ  H_{v}\circ  \hat h_x : \,\w^{ss}(x) \to \w^{ss}(y)=\w^{ss}(x).
$$ 
For $y\in \w^{ss}(x)$ we have that $\hat h_y$ and $\hat h_x$ are related by
the translation holonomy $H_{\ti v}$
$$ \hat h_x \circ  (\hat h_{y})^{-1}  = H_{\ti v}:\, W^{ss}(h(y)) \to W^{ss}(h(x)),
$$
where $\ti v=\hat h_x(y)-h(y)\in E^{ws}$.  We conclude that  
\begin{equation} \label{Hv}
\H_{x,y_n}  \;\text{ converges to }\; \H_{\hat v}:=(\hat h_x)^{-1}\circ  H_{\hat v}\circ  \hat h_x : \w^{ss}(x) \to \w^{ss}(x) =\w^{ss}(y),
\end{equation} 
where $H_{\hat v}=H_{\ti v} \circ H_v$ is the translation by $\hat v=v+\ti v=\hat h_x(y)- h (x)$.
\vskip.05cm 

Since $\h_y$ depend continuously on $y$, using  \eqref{ss hol SR} and \eqref{Hv}
 we obtain that the corresponding lifted holonomies $ \bar H_{x,y_n}$ converge to
 a sub-resonance polynomial $ P_{x,y}\in \s_{x,y}$,
 $$  P_{x,y}=  \h_{y} \circ  \H_{\hat v} \circ  \h_x^{-1} =
  \h_{y} \circ  (\hat h_y)^{-1}\circ H_{\hat v} \circ (\hat h_x) \circ  \h_x^{-1}  :\;  \E_x^{ss} \to \E_y^{ss}.
 $$
  
Now we lift  $\hat h_x$ to $\w ^{ss}(x)$ to $\E_x$  using coordinates $\h_x$
  $$
  \bar h_x = \hat h_x \circ \h_x^{-1} :\; \E_x^{ss} \to W^{ss}(h(x)),
  $$
and  conjugate the translation $H_{\hat v}$  by $\bar h_x$ to obtain
$$ 
\bar H_{\hat v}=   (\bar h_x)^{-1}\circ H_{\hat v} \circ \bar h_x =
  \h_{x} \circ  h^{-1}\circ H_{\hat v} \circ h \circ  \h_x^{-1}=  \h_{x} \circ \h_y^{-1} \circ P_{x,y} :\;  \E_x^{ss} \to \E_x^{ss}.
 $$
Since $P_{x,y} \in \s_{x,y}$ and $\h_{x} \circ \h_y^{-1} :\;  \E_y^{ss} \to \E_x^{ss}$
is a composition of a sub-resonance polynomial in $ \s_{y,x}$ with a translation of $\E_x^{ss}$ we conclude that  
$\bar H_v \in \bar \s_x$, the finite dimensional Lie group  generated by $\s_x$ and the translations of $\E_x^{ss}$.
  Thus $\bar h_x$ conjugates the action of $E^{ss}$ by translations of $W^{ss}(h(x))$ 
with the  continuous action of $E^{ss}$ by elements of $\bar \s_x$. So we get the injective  continuous homomorphism 
 $$
 \eta_x : E^{ss} \to  \bar \s_x \quad\text{given by }\; 
 \eta_x ({\hat v})= \bar H_{\hat v} =(\bar h_x)^{-1} \circ H_{\hat v} \circ \bar h_x.
 $$ 
 which is $C^\infty$ and depends continuously on $x$ in $C^\infty$ topology. 
 This yields  that $\hat h^{-1}_x$ and $\hat h_x$ are also $C^\infty$ diffeomorphisms that
  depend continuously on $x$ in $C^\infty$ topology. 
   
This proves that the component $h^{ss}$ is uniformly $C^\infty$ along the leaves of  $\w^{ss}$, 
as it is easy to see that $\hat h_x = h^{ss} |_{\w^{ss}(x)}$ under a local identification of 
$W^{ss}(h(x))$ with $E^{ss}$.
Since  $h^{ss}$ is also locally constant along the transversal leaves $\w^{u+ws}$, it is as regular 
 along these leaves as they are. Since $\w^u$ has iniformly $C^\infty$ leaves and $\w^{ws}$
 has uniformly $C^q$, the leaves of the joint foliation $\w^{u+ws}$ are uniformly $C^{q}$.
 By Journ\'e lemma we conclude that  $h^{ss} \in C^{q} (\T^d)$.
 \vskip.2 cm
 
{\bf (2)$\implies$(1$'$).}\, Since $h^{ss}$ is uniformly $C^\infty$ along the leaves of  $\w^{ss}$
and  $\hat h_x =h^{ss}|_{\w^{ss}}$, it follows that
the maps $\hat h_x$ are also uniformly $C^\infty$ along the leaves of  $\w^{ss}$. It is easy to see from
the definition  \eqref{h hat} that $\hat h_x$ conjugates the holonomies of $\w^{ws}$
inside $\w^{s}$ with corresponding linear holonomies:
$\H_{x,y} =(\hat h_y )^{-1}\circ H_{x,y} \circ \hat h_x$ and hence $\H_{x,y}$ are uniformly $C^\infty$.
  
     \vskip.2cm
{\bf The case when $h(\w^{ss})=W^{ss}$.} $\;$
When $h(\w^{ss})=W^{ss}$,  it is clear  from \eqref{h hat}
that $h|_{\w^{ss}}=\hat h_x=h^{ss}|_{\w^{ss}}$ and hence (2) implies
that  $h$  is a uniformly $C^{\infty}$ diffeomorphism along $\w^{ss}$. 
Conversely, if $h(\w^{ss})=W^{ss}$ and  $h$  is a uniformly $C^r$  diffeomorphism along $\w^{ss}$ with $r>r_{ss}(L)$, then so are the holonomies of $\w^{ws}$ as they are conjugate
to the linear holonomies: $\H_{x,y} =(h|_{\w^{ss}(y)} )^{-1}\circ H_{x,y} \circ h|_{\w^{ss}(x)}$.
This yields (1).

  \vskip.2 cm
 
{\bf (2)$\implies$(3).}\, The argument is almost identical to the proof of (3) of Theorem \ref{th strong}.
 We map the joint foliation $\w^{u+ws}=h^{-1}(W^{u} \oplus W^{ws})$
  to  $W^{u} \oplus W^{ss}$ by the map
 $$\tilde h (x)=h(x)-\d^{u+ws}(x)=x+\d^{ss}(x) =x^{u+ws}+ h^{ss}(x): \,\T^d \to \T^d.
$$ 
This map  is $C^0$ close to the identity and is  in $C^{q} (\T^d) $ since  $h^{ss}$ is $C^{q} (\T^d)$  by (2).
Invertibility of its derivative $D\tilde h$ follows as in Theorem \ref{th strong} from the fact that 
 $\hat h_x = h^{ss} |_{\w^{ss}(x)} :\w^{ss}(x) \to W^{ss}(h(x))$
 is a $C^\infty$ diffeomorphism by the proof of (2).
 \vskip.2cm
{\bf (3)$\implies$(4).}\, This follows by intersecting the $C^{q}$ foliation $\w^{u+ws}$ with 
uniformly $C^\infty$ leaves of $\w^s$.


\section{Proof of Theorem \ref{strong+weak}} \label{proof strong+weak}

We note that (2)$\iff$(2$'$) does not require dense leaves. 
\vskip.1cm

{\bf (2)$\implies$(2$'$).} Since $h(\w^{u})=W^{u}$, the component $h^s$ is locally constant along $\w^u$, and hence uniformly $C^\infty$ along $\w^u$. Together with (2), this yields that $h^s$  is in  $C^\infty(\T^d)$.
\vskip.1cm

{\bf (2$'$)$\implies$(2).}  Since $h(\w^{s})=W^{s}$
we have $h|_{\w^{s}}=h^s|_{\w^{s}}$ under a local identification of $W^s$ and $E^s$, and hence a uniformly $C^\infty$ diffeomorphism.
\vskip.1cm

{\bf (1)$\implies$(2$'$)} follows by combining Theorems \ref{th strong} and \ref{th weak fol}.
 Indeed, $h(\w^{ss})=W^{ss}$ is Theorems \ref{th strong}(1) and hence yields Theorems \ref{th strong}(5), 
 so that $h^{ws}$  is in  $C^\infty(\T^d)$ and a $C^q$ diffeomorphism along $\w^{ws}$ with some $q>1$.
 The second assumption in Theorem \ref{strong+weak}(1) is Theorem \ref{th weak fol}(1) with $r=\infty$
 and hence yields Theorem \ref{th weak fol}(2) with $r=\infty$, so that $h^{ss}$ 
  is in  $C^\infty(\T^d)$ and a $C^\infty$ diffeomorphism along $\w^{ss}$.  Hence $h^s=h^{ws}+h^{ss}$ is in $C^\infty(\T^d)$ and a diffeomorphism along $\w^s$.
\vskip.1cm

{\bf (2$'$)$\implies$(3).}\, The argument is almost identical to the proof of (3) of Theorem \ref{th strong}.
 We map the  foliation $\w^{u}$ to  $W^{u}$ by the $C^\infty$ diffeomorphism
 $$\tilde h (x)=h(x)-\d^{u}(x)=x+\d^{s}(x)=x^{u}+ h^{s}(x): \,\T^d \to \T^d.
$$ 
This map  is $C^0$ close to the identity and is  in $C^{\infty} (\T^d) $ since  $h^{s}$ is in $C^{\infty} (\T^d)$  by (2).
Invertibility of its derivative $D\tilde h$ follows as in Theorem \ref{th strong} from the fact that 
 $h^{s}$ is a diffeomorphism along $\w^{s}$.

\vskip.1cm

{\bf (3)$\implies$(2).}\, This implication requires only that $W^u$ has dense leaves,  which is always true for a hyperbolic automorphism.
 The proof highlights usefulness of conjugacy to a linear foliation. It is an easy version of the holonomy 
 argument where normal form polynomials are replaced by isometries. 

Let $\phi$ be a $C^{\infty}$ diffeomorphism such that $\phi(\w^u)=W^u$ and let $\w'=\phi(\w^s)$.
Since $W^u$ is linear, its holonomies between $W^s$ leaves are isometries with respect 
to the standard metric on $\T^d$. For each $x\in \T^d$ both $T_xW^s$ and $T_x\w'$ are transverse
to $E^u=TW^u$, and so we can identify them by the projection along  $E^u$. 
This defines a continuous Riemannian metric $g'$ on $T\w'$ for which holonomies of $W^u$ between $\w'$ leaves are isometries. Moreover, since the leaves of $\w^s$ are uniformly $C^\infty$, so are the leaves of  $\w'$,
and hence $g'$ is uniformly $C^\infty$ on the leaves of $\w'$. Then $g=\phi^{-1}_*(g')$ defines a
continuous Riemannian metric  on $T\w^s$ which is uniformly $C^\infty$ along the leaves of $\w^s$
and with respect to which the holonomies of $\w^u$ are isometries. 

We note that for each $x$ the isometries of $(\w^s(x),g)$ are $C^\infty$ diffeomorphisms of $\w^{s}(x)$
and form a finite dimensional Lie group $G_x$. Since we have that  the holonomies of $\w^u$ are isometries, 
repeating the holonomy argument as in the proof of Theorem \ref{th strong}, we see that $h|_{\w^s(x)}$ conjugates the action of $E^s$ by translations of
 $W^s(h(x))$ to a continuous action of $E^s$ by isometries in $G_x$. This yields that $h|_{\w^s(x)}$ are
 uniformly $C^\infty$ diffeomorphisms.


\section{Proof of Theorem \ref{symp}} \label{proof symp}

\vskip.1cm
{\bf (1)$\implies$(2,\,3,\,4)} is clear. 
\vskip.1cm
{\bf (3)$\implies$(2).} Let $0<\rho<1$ be the largest absolute value of eigenvalues of $L$ on $E^{ss}$
and $0<\rho'<1$ be the smallest  absolute value of eigenvalues of $L$ on $E^{ws}$.
If $f$ is $C^1$-close to $L$ and $y\in \w^{ss}(x)$, then $\dist(f^nx,f^ny)\le C(\rho+\e)^n$ for all $n\in \N$.
If $h$ is $\a$-H\"older, then  for all $n\in \N$ we have
$$\dist(L^nh(x),L^nh( y))= \dist(h(f^nx),h(f^ny)) \le C'\dist(f^nx,f^ny)^\a \le C'C^\a(\rho+\e)^{\a n}.$$
If $\a$ is sufficiently close to 1 so that $(\rho+\e)^{\a }<\rho'$, this implies that $h(y)\in W^{ss}(h(x))$
and thus $h(\w^{ss}(x))\subset W^{ss}(h(x))$.  Since $h$ is a homeomorphism, $ h(\w^{ss}(x))$ contains an open ball in  $W^{ss}(h(x))$, and then iterating by $f$ yields $h(\w^{ss}(x))= W^{ss}(h(x))$.

\vskip.1cm
{\bf (4)$\implies$(2)} under the density of leaves assumption for Lyapunov subfoliations  of $W^{ws+wu}$
follows from Theorem \ref{th strong}\,(4)$\implies$(1) applied to $f$ for $\w^{ws}$ and to $f^{-1}$ for $\w^{wu}$,
yielding $h(\w^{ss})=W^{ss}$ and  $h(\w^{uu})=W^{uu}$ respectively.
\vskip.1cm
{\bf (2)$\implies$(1).} Applying Theorem \ref{th strong}\,(1)$\implies$(3) we obtain that the bundle $\E^{u+ss}$
is $C^\infty$. Similarly using $f^{-1}$ we obtain that $\E^{s+uu}$ is $C^\infty$ and hence so is
the intersection $\E^{s+uu}\cap \E^{u+ss}=\E^{ss+uu}$. Since $\E^{ss+uu}$ and $ \E^{ws+wu}$
 are symplectic orthogonal, we conclude that $\E^{ws+wu}$ and hence  the corresponding foliation
 $\w^{ws+wu}$ are $C^\infty$. This is the only place where we use the assumption that $f$ preserves a $C^\infty$ symplectic form. Since the foliation $\w^{ws+wu}$ is $C^\infty$, by intersecting it with $\w^s$ and $\w^u$ we obtain 
 that $\w^{ws}$ and $\w^{wu}$ are their respective uniformly $C^\infty$ subfoliations.
 Thus we obtain that Theorem \ref{strong+weak}(1) is satisfied for both $f$ and $f^{-1}$ and hence yields (2$'$)
in each case. Combining them we conclude that $h=h^s+h^u$ is in $C^{\infty} (\T^d)$ with invertible derivative.


\end{document}